\documentclass[11pt]{article}
\usepackage{vmargin}
\usepackage{times}

\usepackage[T1]{fontenc}
\usepackage{epsfig}
\usepackage{graphicx}
\usepackage{amssymb}
\usepackage{amsmath}
\usepackage[utf8]{inputenc}
\usepackage{amsfonts}
\usepackage{amsthm}
\usepackage{paralist}
\usepackage{stmaryrd}

\usepackage{bbm}
\usepackage{bm}

\usepackage{microtype}

\usepackage[colorlinks=true]{hyperref}
\usepackage{cleveref}
\usepackage[dvipsnames]{xcolor}

\theoremstyle{plain}
\newtheorem{corollary}{Corollary}[section]
\newtheorem{theorem}[corollary]{Theorem}
\newtheorem{lemma}[corollary]{Lemma}
\newtheorem{proposition}[corollary]{Proposition}
\newtheorem*{theorem*}{Theorem}
\newtheorem*{lemma*}{Lemma}

\theoremstyle{definition}
\newtheorem{algorithm}{Algorithm}
\newtheorem{definition}[corollary]{Definition}
\newtheorem*{definition*}{Definition}
\newtheorem{Hypo}{Assumption}

\theoremstyle{remark}

\newtheorem*{remark}{Remark}

\newcommand{\R}{\mathbb{R}}

\newcommand{\proba}{\mathbb{P}}
\newcommand{\N}{\mathbb{N}}
\newcommand{\E}{\mathbb{E}}
\newcommand{\Var}{\mathbb{V}}

\newcommand{\T}{\mathcal{T}}
\newcommand{\TSB}{\Upsilon}
\newcommand{\M}{\mathfrak{p}}

\newcommand{\D}{\mathcal{D}}
\newcommand{\V}{\mathcal{V}}

\newcommand{\Dn}{{\D^n}}
\newcommand{\Pn}{{\P^n}}

\newcommand{\K}{\mathbb{K}}

\newcommand{\X}{\mathcal{X}}
\newcommand{\1}{\mathbf{1}}
\renewcommand{\P}{\mathcal{P}}
\newcommand{\n}{\aleph} 

\renewcommand{\d}{{\mathbf d}}
\newcommand{\p}{{\mathfrak p}}
\newcommand{\e}{\varepsilon}
\newcommand{\s}{\sigma}

\newcommand{\ply}{\Rightarrow}
\newcommand{\OmegaD}{\Omega_{\D}}
\newcommand{\OmegaP}{\Omega_{\P}}
\newcommand{\OmegaT}{\Omega_{\Theta}}

\DeclareMathOperator*{\limit}{\longrightarrow}

\DeclareMathOperator*{\SBB}{SB}

\newcommand*{\GH}{\text{GH}}
\DeclareMathOperator*{\supp}{supp}

\newcommand*{\GP}{\text{GP}}
\newcommand*{\WGP}{\text{WGP}}
\newcommand*{\WGH}{\text{WGH}}
\newcommand*{\GHP}{\text{GHP}}
\newcommand*{\WGHP}{\text{WGHP}}

\DeclareMathOperator*{\argmin}{argmin}

\begin{document}
\title{Limit of trees with fixed degree sequence} 
\author{Arthur Blanc-Renaudie}
\date{\today}
\maketitle
 \begin{abstract}  We show, under natural conditions, that uniform rooted trees with fixed degree sequence converge after renormalization toward
inhomogeneous continuum random trees (ICRT). We also provide a sharp upper-bound for the tail of their heights. We also extend our results to $\P$-trees, ICRT, and trees with random degree sequence. 
In passing we confirm a conjecture of Aldous, Miermont, and Pitman \cite{ExcICRT} stating that L\'evy trees are ICRT with random parameters. 
\end{abstract}
\section{Introduction} 
\subsection{Main results}
Trees with fixed degree sequences are universal models related to many others: Galton--Watson trees, $\P$-trees, random bipartite plane maps with prescribed faces \cite{FixedMap,FixedMap2,FixedMap3}, the critical configuration model  \cite{surplus}\dots The main goal of this paper is to study the geometry of those trees when equipped with their shortest-path distance. In a second part, we will discuss some analogous results and consequences for other models of trees: $\P$-trees, ICRT, and trees with random degree sequences. 

Let  $\{V_i\}_{i\in \N}$ be a set of vertices. Given an integer sequence $\D=(d_i)_{1\leq i \leq n}$ with $\sum d_i=n-1$, we consider a uniform rooted tree $T^\D$ among all trees such that for every $i$, $V_i$ has $d_i$ children. We call this random tree a uniform tree with fixed degree sequence $\D$, or for short a $\D$-tree. For convenience, we always assume that $\D$ is non-increasing.

Let $\OmegaD$ be the set of possible degree sequences. For a degree sequence $\D=(d_1,\dots,d_n)\in \OmegaD$, we look at the number of vertices, leaves, vertices with 1 child, and vertices with at least 2 children:
\[ \n^\D:=n \quad ; \quad  \n^\D_0:=\#\{i, d_i=0\} \quad ; \quad \n^\D_1:=\#\{i, d_i=1\} \quad ; \quad \n^\D_{\geq 2}:=\#\{i, d_i\geq 2\}. \]
The variance of the degree plays a major role in the geometry of $\D$-trees (recall that $\sum d_i=n-1$). Let 
\[ \sigma^\D:= \sum_{i=1}^n d_i(d_i-1).\]

We prove that, in great generality, $\D$-trees converge toward ICRT (inhomogeneous continuum random trees) introduced by Aldous, Camarri, Pitman \cite{IntroICRT1,IntroICRT2}. ICRT are $\R$-trees (loopless geodesic metric spaces) parametrized by a sequence $\Theta$ with usually
\begin{equation} \text{(a)}\quad  \sum_{i=0}^\infty \theta_i^2=1 \quad ; \quad \text{(b)} \quad \theta_1\geq \theta_2\geq \dots \quad ; \quad \text{(c)} \quad \theta_0=0 \text{ or } \sum_{i=1}^\infty \theta_i=\infty.  \label{eq:ThetaSet} \end{equation}
We write $(\T^\Theta,\d^\Theta)$ for the ICRT of parameter $\Theta$, and let $\p^\Theta$ denote the natural probability measure on ICRT from \cite{ICRT1}. We refer to Section \ref{sec:defICRT} for detailed definitions. 

Let $(\Dn)_{n\in \N}=(d^n_1,d^n_2,\dots d^n_n)_{n\in \N}\in \OmegaD^\N$ be a sequence of degree sequence. (The fact $\Dn$ corresponds to exactly $n$ vertices is unnecessary but yields better notations.)  To simplify the notations we will usually use the superscripts $n$ instead of $\Dn$. We assume: 
\begin{Hypo}[$\Dn \ply \Theta$]\label{Hypo2}
  \label{Hypo2} $d^n_1/\n^n\to 0$.  $\n_{0}^n\to \infty$. And for all $i\geq 1$, $d_i^n/\sigma^n\to \theta_i$.
\end{Hypo} 
We also let $\d^\D$ denote the shortest-path distance on $T^\D$. And we consider for every $n\in \N$, a measure $\p^n$ on $\{1,\dots,n\}$, such that $\p^n\to 0$ uniformly. (e.g. uniform on the leaves/vertices$\dots$) The next result describes the distances between random vertices:
\begin{theorem} \label{D_GP_T} If $\Dn\ply \Theta$, $\p^n\to 0$ uniformly, and \eqref{eq:ThetaSet} holds, then the following convergence holds for the weak Gromov--Prokhorov (GP) topology (see Appendix \ref{GPdef} for definition of the topology):
\[ \left ( T^n,(\s^n/n) d^n,\M^n \right) \limit^{\WGP} (\T^\Theta,d^\Theta,\p^\Theta).  \]
\end{theorem}

Theorem \ref{D_GP_T} gives a global picture of all natural GP convergences for $\D$-trees. To see why, let us discuss our assumptions. First, it is natural to expect that the number of leaves $\n_{0}^n$ diverges. Then, one may want to remove the assumption $d_1^n/n\to 0$. In other words, some vertices may have a macroscopic degree.  In that case,  we will see in Section \ref{sec:Ptree} that the typical distances between vertices stay finite and prove a kind of limit toward $\P$-trees, a discrete model introduced by Aldous, Camarri, and Pitman \cite{IntroICRT1,IntroICRT2}. Moreover, up to subsequence extraction, we may always find a sequence $\Theta$ such that $d_i^n/\sigma^n\to \theta_i$. Also when $\Dn\ply \Theta$, (b) from \eqref{eq:ThetaSet} holds, and we may chose $\theta_0$ to get (a). Finally, when one removes (c) from \eqref{eq:ThetaSet}, $\p^\Theta$ cannot properly be defined so Theorem \ref{D_GP_T} cannot hold, but we will prove in Section \ref{Sec:CVfirstbranches} that $\Dn$-trees still converge toward $\Theta$-ICRT for a weaker topology. See Section \ref{subsec:RandomDegree} for further technical details on this case distinction, including random degree sequences.

While the GP topology describes distances between random vertices, several important quantities depend on all vertices, for instance the height, the diameter\dots To this end, one usually uses the Gromov--Hausdorff--Prokhorov  (GHP) topology (see Appendix \ref{GHPdef}). To obtain the GHP convergence of $\D$-trees, we need a tightness assumption. For every $l\in \R^+$, $\D \in \OmegaD$ let 
\[ \psi^\D(l):=l\sum_{i=1}^{\n^\D} \frac{d^\D_i-1}{\sigma^\D}(1-e^{-d^\D_il/\sigma^\D}). \]
\begin{Hypo} \label{Hypo3} \label{D_Tight_GHP}
\[ \lim_{y\to +\infty} \limsup_{n\to +\infty} \int_{y}^{\sigma^n} \frac{dl}{\psi^{n}(l)} =0. \]
\end{Hypo} 
\begin{theorem} \label{D_GHP_T} Under the same setting as Theorem \ref{D_GP_T}, if furthermore Assumption \ref{D_Tight_GHP} holds, then the following convergence holds for the weak Gromov--Hausdorff--Prokhorov (GHP) topology
\[ \left ( T^n,(\s^n/n) d^n,\M^n \right) \limit^{\WGHP} (\T^\Theta,d^\Theta,\p^\Theta).  \]
\end{theorem}
Assumption \ref{D_Tight_GHP} is likely near optimal as it coincides with the necessary and sufficient condition $\int^\infty \frac{dl}{\psi(l)}<\infty$ for the compactness of L\'evy trees (see Duquesne, Le Gall \cite{Duquesne1,Duquesne2}) and of ICRT \cite{ICRT1}. In some sense, $\psi^\D$ can thus be seen as an analog of the Laplace exponent of L\'evy trees for $\D$-trees. 

The last main result is a near optimal upper-bound for the tail of the height of $\D$-trees. We conjecture the bound to be optimal up to multiplicative constants and additive logarithmic terms. 
\begin{theorem} \label{thm:D_Height} There exists $c,C>0$ such that for every $\D\in \OmegaD$ and $x\in \R^+$ we have
\[ \proba\left (c\frac{\s^{\D}}{\n^\D}H(T^{\D})>x+\int_{1}^{\sigma^\D}\frac{dl}{\psi^\D(l)}\right)\leq Ce^{-c\psi^\D(x)} . \]
\end{theorem}
The bound matches with the results of Addario--Berry, Devroye, Janson \cite{TailsG} and of Kortchemski \cite{TailsS} for Galton--Watson trees and L\'evy trees in respectively the brownian and the stable case. Let us also mention that independently, Addario--Berry and Donderwinkel \cite{Stupid} also proved a bound on the height for $\D$-trees which is good when $\sigma^\D=O(\sqrt{\n^\D})$, that is when all degrees are small, but fails to get the right multiplicative order for the height in many cases. 

\subsection{Background}
Since the celebrated works of Aldous \cite{Aldous1,Aldous2,Aldous3}, scaling limits of random trees/graphs are central to many studies. Aldous notably proved the convergence of Galton--Watson trees with finite variance toward the Brownian tree, a universal limit for numerous models of trees with height of order $\sqrt{n}$. 

For smaller trees, one usually finds other limits. Among these there are two important models: On the one hand, Le Gall and Le Jan \cite{IntroLevy1,IntroLevy2}, and Le Gall and Duquesne \cite{Duquesne1,Duquesne2} extensively studied L\'evy trees, which notably appears as the limits of Galton--Watson trees with infinite variance. On the other hand, Aldous, Camarri, and Pitman \cite{IntroICRT1,IntroICRT2} introduced ICRT as the limits of $\P$-trees. Also, Aldous, Miermont, and Pitman \cite{ExcICRT} conjectured that L\'evy trees are equal in distribution to ICRT with a random $\Theta$. In Section \ref{sec:randomD}, we show that this conjecture hold whenever the L\'evy trees are GP limits of Galton--Watson trees. This sugest that ICRT are a universal limit for many trees.


To tackle this problem, it is natural to consider one of the most universal model of trees: $\D$-trees. This was first noted by Broutin and Marckert  \cite{Broutin}, who studied $\D$-trees in the case $\sigma^n=O(\sqrt{n})$, where the limit is always the Brownian tree. Marzouk \cite{FixedMap,FixedMap2,FixedMap3} adapted their method to study bipartite planar maps with fixed degree sequence, which are related to $\D$-trees thanks to the Bouttier--Di~Francesco--Guitter bijection \cite{BDG} and the Janson--Stef\'anson bijection \cite{JS}. 
However, this approach  fails to capture the geometry of $\D$-trees in the general case. Instead, we use the first approach of Aldous \cite{Aldous1} based on stick-breaking constructions, but with other algorithms.

Lastly, let us mention that our results help to study the connected components of critical multiplicative graphs and  configuration models. (See the follow-up paper \cite{surplus}.)

\paragraph{Plan of the paper:}
In Section \ref{Sec:CVfirstbranches}, we present our stick-breaking constructions for $\D$-trees and ICRT and prove the convergence of the first branches of $\D$-trees. We then deduce Theorem \ref{D_GP_T} in Section \ref{GPsection}. Next, in Section \ref{GHPsection} we show Theorems \ref{D_GHP_T} and \ref{thm:D_Height}. In Section \ref{sec:Ptree} we deal with the special case $d^n_1/n\nrightarrow 0$ where $\D$-trees converges toward $\P$-trees. Then, in Section \ref{sec:randomD} we extend our main results for $\P$-trees, ICRT, and for random degree sequence.  Section \ref{sec:Ptree} and \ref{sec:randomD} can be read after Section \ref{1.2}. Appendix \ref{sec:Topology} recall the definitions of the GP, GH, GHP topologies. And Appendix \ref{sec:stable}  present simple applications of our main results to some stable cases.

\paragraph{Notations:} In this paper similar variables for $\D$-trees, $\P$-trees, $\Theta$-ICRT share similar notations. To avoid ambiguity, the models that we are using and their parameters are indicated by superscripts $\D,\P,\Theta$, $\Dn$, $\P^n$, $\Theta^n$ or simply $n$. We often drop those superscripts when the context is clear.

\section{Convergence of the first branches} \label{Sec:CVfirstbranches}
This section is organized as follows. We first define ICRT, then we present the Foata--Fuchs \cite{FoataFuchs} stick-breaking construction of $\D$-trees. We then prove that this construction converges in some sense (Proposition \ref{pro:First_branches_D}) toward the construction for ICRT.
\subsection{Definition of ICRT} \label{2.4} \label{sec:defICRT}
We first define a generic stick breaking construction.  It takes for input two sequences in $\R^+$ called cuts ${\textbf y}=(y_i)_{i\in \N}$ and glue points ${\textbf z}=(z_i)_{i\in \N}$, which satisfy
\begin{equation*} \forall i<j,\ \ y_i<y_j \qquad ; \qquad y_i\limit \infty \qquad ; \qquad \forall i\in \N,\ \ z_i\leq y_i, \label{2609} \end{equation*}
and creates an $\R$-tree by recursively "gluing" the segment $(y_i,y_{i+1}]$ at position $z_i$,  or rigorously, by constructing  a consistent sequence of distances $(d_n)_{n\in \N}$ on $([0,y_n])_{n\in \N}$. \begin{algorithm} \label{Alg1} \emph{Generic stick-breaking construction of $\R$-tree.}
\begin{compactitem}
\item[--] Let $d_0$ be the trivial metric  on $[0,0]$.
\item[--] For each $i\geq 0$ define the metric $d_{i+1}$ on $[0, y_{i+1}]$ such that for each $x\leq y$: 
\[ d_{i+1}(x,y):=
\begin{cases} 
d_{i}(x,y) & \text{if } x,y\in [0, y_i] \\
d_{i}(x,z_i)+|y-y_i| & \text{if } x \in [0, y_i], \, y \in (y_i, y_{i+1}] \\
|x-y| &  \text{if } x,y\in (y_i, y_{i+1}],
\end{cases} \]
where by convention $y_0:=0$ and $z_0:=0$.
\item[--] Let $d$ be the unique metric on $\R^+$ which agrees with $d_i$ on $[0, y_i]$ for each $i\in \N$.
\item[--] Let $\SBB({\textbf y},{\textbf z})$ be the completion of $(\R^+,d)$.
\end{compactitem}
\end{algorithm}

Now, let $\OmegaT$ be the space of sequences $(\theta_i)_{i\in \{0\}\cup \N}$ in $\R^+$ such that $\sum_{i=0}^\infty \theta_i^2=1$ and $\theta_1\geq \theta_2\geq \dots$ For every $\Theta\in \OmegaT$, the $\Theta$-ICRT is the random $\R$-tree constructed as follows:
\begin{algorithm} \label{ICRT} \emph{Construction of $\Theta$-ICRT (from \cite{ICRT1})}
\begin{compactitem}
\item[-] Let $(X_i)_{i\in \N}$ be independent exponential random variables of parameter $(\theta_i)_{i\in \N}$.
\item[-] Let $\mu$ be the measure on $\R^+$ defined by $\mu=\theta_0^2 dx+\sum_{i=1}^{\infty} \delta_{X_i} \theta_i$. 
\item[-] Let $(Y_i,Z_i)_{i\in \N}$ be a Poisson point process  on $\{(y,z): y\geq z \geq 0\}$ of intensity $dy\times d\mu$.
\item[-] The $\Theta$-ICRT is defined as $(\T^\Theta,\d^\Theta)=\SBB((Y_i)_{i\in \N},(Z_i)_{i\in \N})$. (see Algorithm \ref{Alg1})
\end{compactitem}
\end{algorithm}
Then we recall the probability measure on ICRT introduced in \cite{ICRT1}. To simplify our expressions, note that $\mu^\Theta[0,\infty]=\infty$ holds if and only if either $\theta^\Theta_0>0$ or $\sum_{i=1}^\infty \theta^\Theta_i=\infty$, (which is \eqref{eq:ThetaSet} (c)).

\begin{definition}[Proposition 3.2 from \cite{ICRT1}]  Let $\Theta\in \OmegaT$ with $\mu^\Theta[0,\infty]=\infty$.  
Almost surely, $\sum_{i=1}^n \delta_{Y_i}/n$ converges weakly on $\T^\Theta$  as $n\to \infty$ toward a probability measure $\p^\Theta$.
\end{definition}
Finally let us recall the criterium for the compactness of ICRT in terms of $\psi^\Theta:l\mapsto l\E[\mu^\Theta[0,l]]$:
\begin{theorem}[Theorem 3.3 from \cite{ICRT1}] The ICRT is almost surely compact if and only if 
\[ \int_1^\infty \frac{dl}{\psi^\Theta(l)}<\infty. \] 
\end{theorem}
\subsection{Stick-breaking construction for $\D$-trees} \label{1.2}
In this section we present the Foata--Fuchs \cite{FoataFuchs} stick-breaking construction of $\D$-trees (see also \cite{FoataFuchs2} for a recent expository article), which is at the center of this paper. We use the next conventions: For every graph $G=(V,E)$ and edge $e=\{v_1,v_2\}$, let $G\cup e$ denote the graph $(V\cup \{v_1,v_2\},E\cup \{e\})$. We say that a vertex $v\in G$ if $v\in V$. Also, for every degree sequence $\D$ let $L^\D_1$, $L^\D_2$, \dots denote the leaves (that is the vertices $V_{a_1},V_{a_2}\dots$ with $a_1\leq a_2\leq \dots$ and $d_{a_1}=d_{a_2}=\dots =0$).

\begin{algorithm} \label{D-tree} \emph{Foata--Fuchs \cite{FoataFuchs} stick-breaking construction of a $\D$-tree (see Figure \ref{explore1}).}
\begin{compactitem}
\item[-] Let $A^\D=(A^\D_i)_{1\leq i \leq \n-1}$ be a uniform $\D$-tuple (that is a tuple such that $\forall i\in \N$, $V_i$ appears $d_i$ times).
\item[-] 
Let $T^\D_1:=(\{A_1\},\emptyset)$ then for every $2\leq i \leq \n$ let 
\[ T^\D_i:=\begin{cases} T_{i-1}\cup \{A_{i-1},A_{i}\} & \text{if } A_{i}\notin T_{i-1}.\\T_{i-1}\cup \{A_{i-1},L_{\inf\{k, L_k\notin T_{i-1}\}}\}& \text{if } A_{i} \in T_{i-1} \text{ or } i=\n.
\end{cases} \]
\item[-] Let $T^\D$ denotes the rooted tree $(T_\n,A_1)$.
\end{compactitem}
\end{algorithm}
\begin{figure}[!h] \label{explore1}
\centering
\includegraphics[scale=0.65]{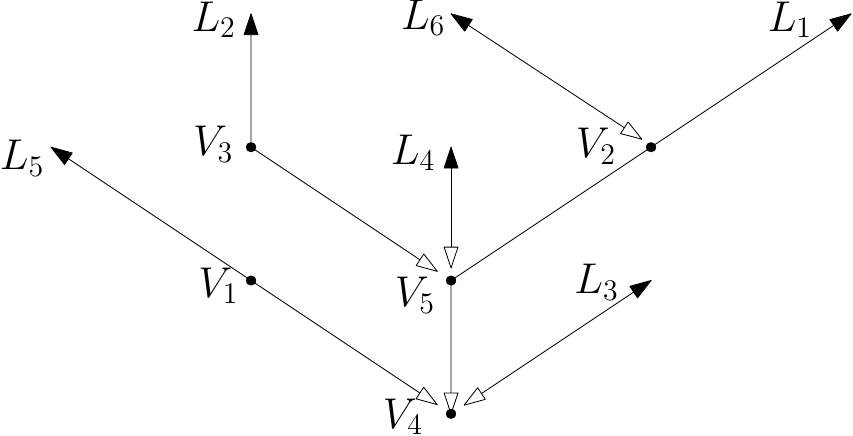}
\caption{Stick breaking construction of a $\D$-tree with $\D=(1,2,1,3,3,0,0,\dots)$ and $(A^\D_i)_{1\leq i \leq \n-1}=(V_4,V_5,V_2 ,V_5,V_3,V_4,V_5,V_4,V_1,V_2)$. The exploration starts at $V_4$ then follows the white-black arrow toward $L_1$, then jumps at $V_5$ to follow the path toward $L_2$ and so on\dots 
Here $Y_1=4, Y_2=6, Y_3=7\dots$, $Z_1=2, Z_2=1, Z_3=2\dots$, $X_1=9, X_2=10, X_3=5\dots$ }
\label{explore1}
\end{figure}%

To compare this algorithm with our construction for ICRT, let us introduce a few notations. We let $Y^\D_1, Y^\D_2\dots $ be the indexes such that $A_i\in \{A_1,\dots, A_{i-1}\}$. Note that there exists exactly $n_0^\D-1$ such indexes as each repetition corresponds to a leaf, with one extra leaf at the end.
Also, for every $1\leq i <\n_0$, let $Z^\D_i$ denote the smallest integer $z$ such that $A_z=A_{Y^\D_i}$. And, for every $i\in \N$, let $X^\D_i$ be the first integer $x$ such that $A_x=V_i$. For ease of notations, for the $i$ where $X_i$, or $Y_i$, or $Z_i$ are not properly defined with the above definitions, we let $X_i=\infty, Y_i=\infty, Z_i=\infty$. We also consider the measure $\mu^\D:= \sum_{i=1}^\infty (d^\D_i-1)\delta_{X^\D_i}$.

The following technical result is central for our proof of the Gromov--Prokhorov convergence, as the positions of the cuts and glue points completely describes the distances between the leaves:
\begin{proposition} \label{pro:First_branches_D}Assume that $\Dn \ply \Theta$ then we have the following weak joint convergence:
\[ (X^n_i\sigma^n/n,Y^n_i\sigma^n/n,Z^n_i\sigma^n/n)_{i\in \N}\to (X^\Theta_i,Y^\Theta_i,Z^\Theta_i)_{i\in \N}\]
\end{proposition}
Finally, to simplify the notations, let $\TSB^\D:=(X^\D_i,Y^\D_i,Z^\D_i)_{i\in \N}$ and $\TSB^\Theta:=(X^\Theta_i,Y^\Theta_i,Z^\Theta_i)_{i\in \N}$. 
\begin{remark} Proposition \ref{pro:First_branches_D} does not require \eqref{eq:ThetaSet} (c), which is neccesary for GP convergence, since it is necessary to define properly $\p^\Theta$.
\end{remark}
 \subsection{Preliminary: continuum modification of $\TSB^\D$} \label{5.2} 
We define here a modification for $\TSB^\D$, which may seem more complex, but is actually easier to study line by line.   (Steps 4 to 7 below will disappear at the limit giving part of the ICRT construction.)
 \begin{algorithm} \label{Alg4} Definition of  $\tilde \TSB^\D$.
\begin{compactitem}
\item[--] Let $(\tilde X_i)_{i\in \N}$ be independent exponential random variables of parameter $(d_i/\sigma)_{i\in \N}$.
\item[--] Let $\tilde \mu$ be the measure on $\R^+$ defined by $\tilde \mu=\sum_{i=1}^{\n} \delta_{ \tilde X_i} \left (d_i-1 \right )/\sigma$. 
\item[--] Let $(\hat Y_i,\hat Z_i)_{i\in \N}$ be a Poisson point process  on $\{(y,z): y\geq z\geq 0\}$ of intensity $dy\times d\tilde \mu$.
\item[--] For every $i\in \N$ with $d_i\geq 1$, let $U_{X,i}$ be uniform in $\{1,\dots, d_i\}$.
\item[--] For every $j\in \N$, let $U_{j}$ be uniform in $\{1,\dots, d_i\}\backslash\{U_{X,i}\}$ where $i$ is the unique index such that $\tilde X_i=\hat Z_j$. ($U_j$ is well defined since $\hat Z_j$ is in the support of $\tilde \mu$, which is $\{\tilde X_i\}_{i:d_i>0}$.)
\item[--] Let  $k_1<  \dots < k_{\n_0-1}$ be the indexes such that $(Z_k, U_k)\notin \{ (Z_j, U_j),j<k\}$. (There is a.s. $\n_0-1$ such indexes since $\n_0-1=\sum_{i:d_i>0} (d_i-1)$.)
\item[--] For every $1\leq i <\n_0$, let  $\tilde Y_i=\bar Y_{k_i}$ and let $\tilde Z_i=\bar Z_{k_i}$. For every $i\geq \n_0$, let $\tilde Y_i=\tilde Z_i=+\infty$.
\item[--] Let $\tilde \TSB^\D=(\tilde X_i, \tilde Y_i, \tilde Z_i)_{i\in \N}$.
\end{compactitem}
\end{algorithm}
Please already note that:
\begin{lemma} \label{psimu}For every $l\geq 0$, we have $\psi^\D(l)=l\E[\tilde \mu^\D[0,l]]$.
\end{lemma}
\begin{proof} It directly follows by linearity of the expectation from the definition of $\psi, \tilde \mu$. 
\end{proof}

Next, for every $f:\R^+\mapsto \R^+$, and $\TSB=(x_i,y_i,z_i)_{i\in \N}$, let $f(\TSB):= ( f(X_i),f(Y_i),f(Z_i))_{i\in \N}.$
Also, for every $\D \in \OmegaD$, let $(E^\D_i)_{1\leq i \leq \n^\D-1}$ be a family of independent exponential random variables of mean $(\sigma^\D/(\n^\D-i))_{1\leq i \leq \n^\D-1}$ and let $f^\D:\R^+\mapsto \R^+$ be a continuous increasing function such that for every $1\leq i \leq \n^\D-1$, we have $f^\D(i)=\sum_{k=1}^i E^\D_k$.

 We can now state the main result of this subsection:
\begin{proposition} \label{timechange} $f^\D(\TSB^\D)$ and $\tilde \TSB^\D$ have the same distribution.
\end{proposition}

\begin{proof} We explicitly construct a coupling between Algorithms \ref{D-tree} and \ref{Alg4} such that a.s. 
\[ (f(X_i), f(Y_i), f(Z_i))_{i\in \N}=(\tilde X_i, \tilde Y_i, \tilde Z_i)_{i\in \N}. \]
 To this end we use the same "starting randomness" for both algorithms. First let 
$I:=\{(i,j)\in \N^2, 1\leq i \leq \n, 1\leq j \leq d_i\}.$ For every $(i,j)\in I$, let $P_{i,j}$ be a Poisson point process on $\R^+$ with rate $1/\s$ and let $E_{i,j}= \min P_{i,j}$. For every $i\in \N$ such that $d_i\geq 1$ let $k_i=\argmin \{ E_{i,j}, 1\leq j \leq d_i\}$. Finally let $I_2=I\backslash\{i,k_i\}_{1\leq i \leq \n}$.

Toward Algorithm \ref{D-tree}, sort $\{(E_{i,j},V_i,j)\}_{(i,j)\in I}$ by the first coordinate as $(t_i,B_i,K_i)_{1\leq i \leq \n-1}$. One can easily check, that $(t_i)_{1\leq i \leq \n-1}$ is independent of $(B_i, k_i)_{1\leq i \leq \n-1}$, that $(t_i)_{i\in \N}$ have the same distribution as $(f(i))_{1\leq i \leq \n-1}$, and that $(B_i,k_i)_{1\leq i \leq \n-1}$ is a uniform permutation of $I$. We omit the details. As a result, $(t_i,B_i)_{1\leq i \leq \n-1}$ have the same distribution as $(f(i), A_i)_{1\leq i \leq \n-1}$. Therefore we may assume that for every $1\leq i \leq \n-1$, $(t_i,B_i)=(f(i),A_i)$. It directly follows that for every $i\in \N$ such that $d_i\geq 1$, 
\begin{equation} f(X_i)=f(\min\{j\in \N,A_j=V_i\})=\min\{f(j), A_j=V_i\}=\min\{E_{i,k},1\leq k \leq d_i \}=E_{i,k_i}. \label{abcd1} \end{equation}
Then by a similar argument,
\begin{equation} \{f(Y_i),f(Z_i)\}_{1\leq i \leq N} =\{E_{i,j},E_{i,k_i}\}_{(i,j)\in I_2}.\label{abcd2} \end{equation}

Toward Algorithm \ref{Alg4}, note that $E_{i,k_i}$ is an exponential random variable of mean $d_i/\sigma$ and that $k_i$ is uniform in $\{1,\dots ,d_i\}$, hence we may assume that 
\begin{equation} \tilde X_i=E_{i,k_i} \quad ; \quad U_{X,i}=k_i. \label{abcd3} \end{equation} 
Then, note that conditionally on $(k_i)_{1\leq i \leq \n}$ and on $(E_{i,k_i} )_{1\leq i \leq \n}$, $\bigcup_{(i,j)\in I_2} P_{i,j}\times \{E_{i,k_i}\} \times \{j\}$ is a Poisson point process on $\R^{+3}$ of intensity
\begin{equation} \sum_{(i,j)\in I_2} \1_{E_{i,k_i}\leq x} dx \times \delta_{E_{i,k_i}} \times \delta_{j}. \label{1203} \end{equation}
Also note that conditionally on $(\tilde X_i)_{1\leq i\leq  \n}$ and on $(U_{X,i})_{1\leq i\leq  \n}$, $\{\hat Y_i,\hat Z_i, U_{Z,i}\}_{i\in \N}$ is also a Poisson point process with the same intensity as in \eqref{1203}. So we may assume that 
\begin{equation} \{\hat Y_i,\hat Z_i, U_{Z,i}\}_{i\in \N}=\bigcup_{(i,j)\in I_2} P_{i,j}\times \{E_{i,k_i}\} \times \{j\}. \label{abcd4} \end{equation}
Therefore, by "keeping" the first point of each line $\{y,z,u\}_{y\in \R^+}$ in both side and by "deleting" the last coordinate, we have,
 \begin{equation} \{\tilde Y_i, \tilde Z_i \}_{1\leq i \leq N}=\{E_{i,j},E_{i,k_i}\}_{(i,j)\in I_2}.\label{abcd4} \end{equation}

Finally by \eqref{abcd1} and \eqref{abcd3} we have $(f(X_i))_{i\in \N}=(\tilde X_i)_{i\in \N}$. Also by \eqref{abcd2}, \eqref{abcd4}, and by monotony of $(Y_i)_{i\in \N}$ and $(\tilde Y_i)_{i\in \N}$ we have $(f(Y_i), f(Z_i))_{i\in \N}=( \tilde Y_i, \tilde Z_i)_{i\in \N}$.\end{proof}

\subsection{Convergence of $\D$-trees toward ICRT: proof of Proposition \ref{pro:First_branches_D}}  \label{5.3}
Fix $(\Dn)_{n\in \N}\in\OmegaD^\N$ and $\Theta\in \OmegaT$. We assume that $\Dn\ply \Theta$, that is $d_1^n/n\to 0$, and $\n^n_0\to \infty$, and $\forall i\in \N$, $d_i^n/\s^n\to \theta_i$. We prove that $(\sigma^n/n)\TSB^n \to \TSB^\Theta$ holds weakly for the joint topology. To this end, we prove a line by line convergence of Algorithm \ref{Alg4} then conclude by using Proposition \ref{timechange}.

\begin{lemma} \label{mass} If $\Dn\ply \Theta$ then jointly in distribution for every $i\in \N$, $\tilde X^\Dn_i\to X^\Theta_i$ and $\tilde \mu^\Dn\to \mu$. (That is jointly for every $l\in \R^+$, $\tilde \mu^\Dn[0,l]\to \mu[0,l]$.)
\end{lemma}
\begin{proof} The convergence of the variables $(\tilde X_i^\Dn)_{i\in \N}$ is immediate from their definitions. So by the Skorokhod representation theorem we may assume that almost surely for every $i\in \N$, $\tilde X^\Dn_i \limit X^\Theta_i$.

We prove that jointly with the previous convergence, $\tilde \mu^\Dn\to \mu^\Theta$ weakly. To this end, for every $\D\in \OmegaD$ and $k\in \N$ we split $x\mapsto \tilde \mu^\D[0,x]$  into
\[ F^{\D,\leq k}: x\mapsto \sum_{i=1}^k  \1_{\tilde X^\D_i\leq x} (d^\D_i-1)/\sigma^\D \quad \text{and} \quad  F^{\D,>k}: x\mapsto \sum_{i=k+1}^\infty  \1_{\tilde X^\D_i\leq x} (d^\D_i-1)/\sigma^\D. \]
And we show that if $(k_n)_{n\in\N}$ is a sequence increasing sufficiently slowly to $+\infty$ then for every $x>0$ (a) almost surely $F^{n,\leq k_n}(x)\to \sum_{i=1}^\infty  \1_{X^\Theta_i\leq x}\theta_i$ and (b) $F^{n,> k_n}(x)\to \theta_0^2 x$ in probability. Summing (a) and (b) yields the desired result. 

Toward (a), recall that  for every $i\in \N$, a.s. $\tilde X^\Dn_i \to X^\Theta_i$, and that $\Dn\ply \Theta$. So if $(k_n)_{n\in\N}$ increases sufficiently slowly to $+\infty$ then by bounded convergence for every $x\in \R^+$ a.s. 
\begin{equation*}  F^{n,\leq k_n}(x)= \sum_{i=1}^{k_n}  \1_{\tilde X^n_i\leq x} (d^n_i-1)/\sigma^n \limit_{n\to \infty}  \sum_{i=1}^\infty \1_{X^\Theta_i\leq x}\theta^\Theta_i.  \end{equation*} 

To prove (b) we use a second moment method. We have for every $x\geq 0$,
\[ \E\left [F^{n,> k_n}(x) \right ]=\sum_{i=k_n+1}^\infty \frac{d^n_i-1}{\sigma^n} \proba\left (\tilde X^n_i\leq x\right )=\sum_{i=k_n+1}^\infty \frac{d^n_i-1}{\sigma^n} \left(1-\exp\left (x\frac{d^n_i}{\sigma^n}\right ) \right ). \]
Then, since $k_n\to \infty$ and since $\Dn\ply \Theta$,
\begin{equation*}
\E\left [F^{n, >k_n}(x) \right ]  \sim  \sum_{i=k_n+1}^\infty x \frac{d^n_i-1}{\sigma^n} \frac{d^n_i}{\sigma^n}  =  x- x\sum_{i=1}^{k_n} \frac{d_i^n(d_i^n-1)}{(\sigma^n)^2}  \to  x-x\sum_{i=1}^{+\infty}\theta_i^2=x\theta_0^2,
\end{equation*}
where the last convergence holds by bounded convergence when $(k_n)_{n\in \N}$ increases sufficiently slowly to $+\infty$.
Also similarly for every $x\geq 0$,
\[ \Var\left [F^{n, > k_n}(x) \right ]\leq \sum_{i=k_n+1}^{+\infty} \left ( \frac{d^n_i-1}{\sigma^n}\right )^2 \proba\left (\tilde X^n_i\leq x\right )\leq \frac{d^n_{k_n+1}}{\sigma^n} \E \left [F^{n,>k_n}(x) \right ] =o(1). \] 
And (b) follows.
\end{proof} 
It directly follows from Lemma \ref{mass}, and since $(\hat Y_i^n,\hat Z_i^n)$ is a Poisson point process  on $\{(y,z)\in \R^{+2}: y\geq z \}$ of intensity $dy\times d\tilde \mu^n$, that the following weak joint convergence holds: 
\begin{equation} \hat \TSB^\Dn:=(\tilde X_i^n,\hat Y_i^n,\hat Z_i^n)_{i\in \N}\limit_{n\to \infty} \TSB^\Theta. \label{2711} \end{equation}
We omit the trivial details, and refer for instance to Lemma 4.24 (ii') of Kallenberg \cite{Kallenberg} for more precision on convergence of Poisson point process.

The next lemma implies that we may replace in \eqref{2711} $(\hat Y_i^n,\hat Z_i^n)_{i\in \N}$ by $(\tilde Y_i^n,\tilde  Z_i^n)_{i\in \N}$.
\begin{lemma} \label{SBtild} Assume that $\Dn\ply \Theta$ then for every $i\in \N$,
\[ \proba\left (\hat Y^n_i=\tilde Y^n_i \text{ and }\hat Z^n_i=\tilde Z^n_i \right )\limit_{n\to \infty} 1. \]
\end{lemma}
\begin{proof} Let for $\D\in \OmegaD$, $m^\D:=\inf( (\tilde Y^\D_i)_{i\in \N} \backslash (\hat Y^\D_i)_{1\leq i \leq N} )$. 
Note that it is enough to show that for every $l\in \R^+$, $\proba (m^n> l)\to 1$. To this end, we lower bound $\proba (m^\D> l)$ for $\D$ and $l$ fixed.

First let us recall some notations introduced in the proof of Proposition \ref{timechange}. Let $I:=\{(i,j)\in \N^2, 1\leq i \leq \n, 1\leq j \leq d_i\}.$ For every $(i,j)\in I$, let $P_{i,j}$ be a Poisson point process on $\R^+$ with rate $1/\s$ and let $E_{i,j}= \min P_{i,j}$. For every $i$ such that $d_i\geq 1$ let $k_i=\argmin \{ E_{i,j}, 1\leq j \leq d_i\}$. Then let $I_2=I\backslash\{i,k_i\}_{1\leq i \leq \n}$.

Now, recall that $\{\tilde Y_i\}_{i\in \N} \backslash \{\hat Y_i\}_{i\in \N}=\bigcup_{(i,j)\in I_2}\{P_{i,j}\backslash\{E_{i,j}\} \}$. So that by an union bound,
\begin{equation}  \proba(m> l) 
\leq \sum_{i=1}^\n \proba \Big (\exists 1\leq j \leq d_i: (i,j)\in I_2, \#(P_{i,j}\cap [0,l])\geq 2 \Big ).
\label{meteo}
\end{equation}
To simplify the notation, let $\sum_{1\leq i \leq \n} S_i$ denote the sum above. We split this sum in two according to "whether $d_i$ is small or large", and we upper bound each part. 

On the one hand, note that if there exists $1\leq j \leq d_i$ such that $(i,j)\in I_2$ and $\#(P_{i,j}\cap [0,l])\geq 2$ then $\#(\bigcup_{1\leq j \leq d_i} P_{i,j}\cap[0,l])\geq 3$ and $d_i\geq 2$. So that for every $1\leq i\leq \n$, 
\begin{equation*} S_i
 \leq  \proba \left (\#\left (\bigcup_{1\leq j \leq d_i} P_{i,j}\cap[0,l] \right )\geq 3 \right ) =1-e^{-ld_i/\sigma} \left (1+\frac{ld_i}{\sigma}+\frac{l^2d^2_i}{2\sigma^2} \right ) \leq \left (l\frac{d_i}{\sigma}\right )^3, \end{equation*} where the last inequality comes from the fact that for every $x\geq 0$, $1-e^{-x}(1+x+x^2/2)\leq x^3$. 
Then for every $\e>0$,
\begin{equation} \sum_{i, d_i\leq \e \sigma} S_i\leq \sum_{i, d_i\leq \e \sigma} \1_{d_i\geq 2} \left (l \frac{d_i}{\sigma}\right )^3 \leq 2l^3\e \sum_{i, d_i\leq \e \sigma}  \frac{d_i(d_i-1)}{\sigma^2}\leq 2l^3\e. \label{elena1} \end{equation}

On the other hand, we have by a similar argument,
\[ S_i\leq d_i \proba \left ( \#(P_{i,1}\cap[0,l]) \geq 2 \right ) \leq d_i l^2/\sigma^2, \]
and when $\e\sigma >1$,
\begin{equation} \sum_{i, d_i> \e \sigma} P_i \leq \sum_{i, d_i> \e \sigma} l^2\frac{d_i}{\sigma^2} \leq \sum_{i, d_i> \e \sigma} \frac{l^2 }{\e\sigma-1} \frac{d_i(d_i-1)}{\sigma^2} \leq \frac{l^2}{\e\sigma-1}. \label{elena2} \end{equation}

Therefore, summing \eqref{elena1} and \eqref{elena2}, we have for every $n\in \N$,
\begin{equation} \proba(m^n> l)\leq 2l^3\e +\frac{l^2}{\e\sigma^n-1}.  \label{elena} \end{equation}
Then as $n\to +\infty$, $\Dn\ply \Theta$ so $\sigma^n \to \infty$ and the right hand term in \eqref{elena} converges to $2l^3\e$. Since $\e$ is arbitrary, $\proba(m^n> l)\to 0$. Since $l$ is arbitrary, this concludes the proof.
\end{proof}
\begin{proof}[Proof of Proposition \ref{pro:First_branches_D}] Recall the definitions of $f^\D$ and $E_i^\D$ introduced above Proposition \ref{timechange}. It  follows from Lemma \ref{SBtild} and  \eqref{2711} that weakly $\tilde \TSB^n \to \TSB^\Theta$. Then by Proposition \ref{timechange}, $f^n(\TSB^n)$ and $\tilde \TSB^n$ have the same distribution. So, by Skorokhod representation theorem we may assume that both the convergence and the equality holds almost surely. Hence, almost surely
\begin{equation} f^n(\TSB^n)\to  \TSB^\Theta. \label{0412} \end{equation}

Next, for $n\in \N$ let $\lambda^n:=  n/\sigma^n$. We have, for every $x\in \R^+$ and $n\in \N$, 
\begin{equation*}\E\left [ f^n\left ( \left \lfloor \lambda^n x \right \rfloor \right ) \right ]=\sum_{i=1}^{ \lfloor \lambda^n x \rfloor } \E[E^n_i] =\sum_{i=1}^{  \lfloor \lambda_n x \rfloor } \frac{\sigma^n}{n-i}= \lfloor \lambda^n x \rfloor  \frac{\sigma^n}{n-O(\lambda^n)}\limit_{n\to \infty} x.
\label{riz1} \end{equation*}
 Similarly,
\[ \Var[ f^n( \lfloor \lambda^n x \rfloor)]=\sum_{i=1}^{ \lfloor \lambda^n x \rfloor} \Var[E_i] =\sum_{i=1}^{ \lfloor \lambda^n x \rfloor} \left ( \sigma^n/(n-i)  \right )^2 \limit 0.\]
Therefore $x\mapsto f^n(\lambda_n x)$ converges in distribution for $\|\cdot \|_{\infty}$ toward the identity on any interval. Then by Skorokhod representation theorem we may assume that this convergence holds almost surely. It directly follows from \eqref{0412} that $(\sigma^n/n)\TSB^\Dn\limit \TSB^\Theta$.
\end{proof} 

To conclude the section let us slightly strengthen Proposition \ref{pro:First_branches_D} by considering $(\1_{X_i=Z_j})_{i,j}$, and $\mu^\D:i \mapsto \sum_{i=1}^{\n}(d_i-1)\delta_{X_i}$. This extension is notably important to study many other object related to $\D$-trees and ICRT introduced in \cite{LoopICRT}.
\begin{proposition} \label{pro:First_branches_D+} We have for the weak topology, jointly with Proposition \ref{pro:First_branches_D}, 
\[ \forall l>0, \mu^n[0,\sigma^n l/n]/\sigma^n \limit \mu^\Theta[0,l] \quad \text{and} \quad (\1_{X_i^n=Z_j^n})_{i,j\in \N}\limit (\1_{X_i^\Theta=Z_j^\Theta})_{i,j\in \N}. \]
\end{proposition}
\begin{proof} For $\mu$, note that by Lemma \ref{mass}, jointly with \eqref{2711}, we can also have $\tilde \mu^n\limit \mu^\Theta$. Moreover, by Skorokhod's representation theorem we may assume that $f^n(\TSB^n)=\tilde \TSB$. In that case,
\begin{equation} \mu^n[0,i]=\sum_{j=1}^n (d^n_j-1) \1_{X^n_i\leq i} = \sum_{j=1}^n (d^n_j-1) \1_{\tilde X^n_i\leq f^n(i)} =\tilde \mu [0,f^n(i)]/\sigma. \label{13/02/10h} \end{equation}
Then recall that $x\mapsto f^n(nx/\sigma^n)$ converges to the identity on any interval, so $\mu$ properly rescaled converges. The indicators can be treated similarly by strengthening \eqref{2711}. (Again it's just that Poisson point process converge when their rates converge.) 
\end{proof}

\section{Gromov--Prokhorov convergence of $\D$-trees} \label{GPsection}
\subsection{Preliminary: vertices with small degree behaves like leaves} \label{6.1}
To prove the GP convergence, it suffices to prove that the distance matrix between random vertices converges (see Lemma \ref{equivGP2}).
To this end, we introduce a generalization of Algorithm \ref{D-tree} which constructs sequentially the subtree spanned by the root and $W_1,\dots, W_i$ where $W=\{W_i\}_{1\leq i \leq \n^\D}$ is an arbitrary permutation of $\V^\D$. Then, we couple it with Algorithm \ref{D-tree} to show that whenever the degrees of $W_1, W_2,\dots $ are small, those subtrees behaves like subtrees spanned by the first leaves (see Lemma \ref{CoupleSpann}). This agrees with the intuition that vertices with small degree behave like leaves. Naturally, this is also true for random vertices, which typically are distinct and have small degree.

\begin{algorithm} \label{Spann} \emph{General stick-breaking construction of $\D$-tree.} 
\begin{compactitem}
\item[-] Let $A^\D=(A_1,\dots, A_{\n-1})$ be a uniform $\D$-tuple. 
\item[-] 
Let $T^{\D,W}_1:=(\{A_1\},\emptyset)$ then for every $2\leq i \leq \n$ let 
\[ T^{\D,W}_i:=\begin{cases} T_{i-1}\cup \{A_{i-1},A_{i}\} & \text{if } A_{i}\notin T_{i-1}, \\T_{i-1}\cup \{A_{i-1},W_{\inf\{k, W_k\notin T_{i-1}\}}\}& \text{if } A_{i} \in T_{i-1} \text{ or } i=\n .
\end{cases} \]
\item[-] Let $T^{\D,W}$ denote the rooted tree $(T_{\n},A_1)$.
\end{compactitem}
\end{algorithm}
\begin{remark} If for every $1\leq i \leq N+1$, $W_i=L_i$ then Algorithm \ref{D-tree} and Algorithm \ref{Spann} follow the exact same steps, hence $T^{\D,W}=T^\D$.
\end{remark}
\begin{proposition} \label{proof_algo2}
For every $\D\in \OmegaD$, and $W$ permutation of $\V^\D$, $T^{\D,W}$ is a  $\D$-tree.
\end{proposition}
 \begin{proof} It is well known that there is exactly $(\n-1)!/(d_1!d_2!d_3!\dots d_\n!)$ (see e.g. \cite{FoataFuchs}) trees with degree sequence $\D$, which is also exactly the number of $\D$-tuples. So it suffices to check that the algorithm provides an injection from $\D$-tuples to trees with degree sequence $\D$. Note that the algorithm creates a graph with exactly $\n-1$ edges and that is connected so is a tree. Moreover, note that if we orient the edges from the root to the leaves, each edges leaving a vertex $V_i$ corresponds to an instant where $A_j=V_i$ so the $T^{\D,W}$ has the desired degree sequence. Finally to see that the algorithm is injective, note that $A_1,A_2,\dots$ first describes the branches from the root to $W_1$, then the branch starting from this branch to $W_2$ (there may be $0$ steps if $W_2$ was seen before $W_1$) then from the previous subtree to $W_3$ and so on. We omit the details, and refer the reader to \cite{FoataFuchs2} where similar proofs for similar algorithms are fully detailed. 
 \end{proof}

Before stating the last result of this section, let us define the relabelling operation. For every graph $G=(V,E)$ and bijection $f:V\to V'$, let $f(G):=(V',\{\{f(x),f(y)\}\}_{\{x,y\}\in E })$.

\begin{lemma} \label{CoupleSpann} Let $\D\in \OmegaD$ and $W=\{W_i\}_{1\leq i \leq \n^\D}$ be a permutation of $\V^\D$. Let $1\leq k \leq N^\D$ and let $f_k:\V^\D\to \V^\D$ be a bijection such that:
\[ \forall 1\leq i \leq k \quad f_k(W_i)=L_i \quad ; \quad \forall V\notin \{W_i\}_{1\leq i \leq k}\cup \{L_i\}_{1\leq i \leq k} \quad  f_k(V)=V. \]
Then for every $1\leq l \leq \n^\D$, (where by convention for $i\in \N$, $d^\D_{V_i}:=d^\D_i$)
\[ \proba \left (f_k\left (T^{\D,W}_{Y^\D_k} \right )\neq T^\D_{Y^\D_k} \right ) \leq \proba( Y^\D_k> l)+l \left (d^\D_{W_1}+\dots d^\D_{W_k}\right)/(\n^\D-1). \]
\end{lemma}
\begin{proof} Note that $f_k(T^{W}_{Y_k})=T_{Y_k}$ whenever $\{W_i\}_{1\leq i \leq k}\cap \{A_1,\dots, A_{Y_k}\}=\emptyset$, since in this case, up to relabelling, Algorithm \ref{D-tree} and Algorithm \ref{Spann} follow the exact same steps. Therefore: 
\begin{align*} \proba \left (f_k(T^{W}_{Y_k})\neq T_{Y_k} \right ) & \leq  \proba( Y_k> l)+\proba (\{W_i\}_{1\leq i \leq k}\cap \{A_1,\dots, A_{l}\}\neq \emptyset), 
\\ & \leq  \proba( Y_k> l)+\sum_{i=1}^l \sum_{j=1}^k  \proba (A_i=W_j),
\\ & =  \proba( Y_k> l)+l \left (d_{W_1}+\dots d_{W_k}\right)/(\n-1). \qedhere
\end{align*}\end{proof}
 \subsection{GP convergence of $\D$-trees toward ICRT: proof of Theorem \ref{D_GP_T}} \label{6.2} 
 To simplify the notations we write for $\D\in \OmegaD$, $\lambda^\D:=\n^\D/\s^\D$.
 
 First recall the assumptions of Theorem \ref{D_GP_T}: $\Dn \ply \Theta$, $\M^n \to 0$, and $\mu^\Theta=\infty$. In this case we have by Proposition \ref{pro:First_branches_D} the following weak joint convergence
\begin{equation} ( ( Y_i^n/\lambda^n )_{i\in \N},(Z_i^n/\lambda^n )_{i\in \N} )  \limit ((Y_i^\Theta)_{i\in \N},(Z_i^\Theta)_{i\in \N} ). \label{1106.23} \end{equation}
Moreover, note that the distance between the leaves of a tree constructed by stick breaking is described by the positions of the cuts and the glue points. More precisely, for every $i<j\in \N$, there exists a measurable function $g_{i,j}: \R^{+2j}\mapsto \R^+$ such that for every $n\in \N$ large enough (such that there is at least $j$ leaves) and $\Theta\in \OmegaT$:
\[  (\d^n/\lambda^n)(L^n_i, L^n_j)=g_{i,j} ( (Y^n_k/\lambda^n)_{1\leq k\leq j}, ( Z^n_k/\lambda^n )_{1\leq k\leq j} ), \]
and
\[ \d^\Theta(Y^\Theta_i, Y^\Theta_j)=g_{i,j}((Y^\Theta_k)_{1\leq k\leq j}, (Z^\Theta_k)_{1\leq k\leq j}).\]
Therefore, by \eqref{1106.23} we have the following weak joint convergence 
\begin{equation}  ((\d^n /\lambda^n)(L_i^{n}, L_j^{n} ) )_{i,j\in \N} \limit (\d^\Theta(Y_i^{\Theta}, Y_j^{\Theta} )_{i,j\in \N}). \label{1106.232} \end{equation} 

Now for every $n\in \N$ let $(W^n_i)_{i\in \N}$ be a family of i.i.d. random variables with law $\M^n$. Recall by Lemma \ref{equivGP} and by \cite{ICRT1} Proposition 3.2 that it suffices to prove that we have the next weak joint convergence
\begin{equation} ((\d^n /\lambda^n)(W_i^{n}, W_j^{n} ) )_{i,j\in \N} \limit (\d^\Theta(Y_i^{\Theta}, Y_j^{\Theta} ))_{i,j\in \N}. \label{1106.233} \end{equation}
To this end we use the coupling introduced in Section \ref{6.1} to derive \eqref{1106.233} from \eqref{1106.232}.

Beforehand, note that for every $n\in \N$, $(W^n_i)_{i\in \N}$ is not a permutation of $\V^n$ so we can not directly apply  Lemma \ref{CoupleSpann}. However, since $\M^n \to 0$ we have for every $i,j\in \N$, $\proba(W_i=W_j)\to 0$. Hence, there exists a family $(\tilde W^n)_{n\in \N}=((\tilde W^n_i)_{1\leq i \leq \n^n})_{n\in \N}$ of random permutations of $\V^n$ such that for every $i\in \N$, $\proba(W^n_i=\tilde W^n_i)\to 0$.

We now apply Lemma \ref{CoupleSpann} to those permutations: We have for every $n,k\in \N$ and $l\in \R^+$,
\begin{equation*} \proba \left (f^n_k\left (T^{n,\tilde W^n}_{Y^n_k} \right )\neq T^n_{Y^n_k} \right ) \leq \proba \left ( Y^n_k> \lambda^n l\right )+\E\left [\frac{\lambda^n l}{n-1} \left (d^n_{\tilde W^n_1}+\dots +d^n_{\tilde W^n_k} \right ) \right ], \label{crepe} \end{equation*}
where $f^n_k$ is the relabelling function defined in Lemma \ref{CoupleSpann}. Moreover note that the last term converges to $0$. Indeed since $\M^n \to 0$, for every $i\in \N$, $d^n_{\tilde W^n_i}/\sigma^n\to 0$ in probability. And by bounded convergence the last also convergence holds in expectation since $d^n_{1}/\sigma^n\to \theta_1<\infty$. Therefore for every $k,l$ fixed,
\[ \limsup_{n\to \infty} \proba \left (f^n_k\left (T^{n,\tilde W^n}_{Y^n_k} \right )\neq T^n_{Y^n_k} \right ) \leq \limsup_{n\to \infty} \proba \left ( Y^n_k> \lambda^n  l\right ). \] 
Therefore, since $l$ is arbitrary and since by Proposition \ref{pro:First_branches_D} $\lambda^n Y^n_k\to Y^\Theta_k$ weakly as $n\to \infty$, we have,
\begin{equation} \limsup_{n\to \infty} \proba \left (f^n_k\left (T^{n,\tilde W^n}_{Y^n_k} \right )\neq T^n_{Y^n_k} \right )=0. \end{equation}

Finally, since relabelling does not change the distance in the tree,
\[ d_{TV} \left (  (\d^n(L_i^{n}, L_j^{n} )  )_{i,j\leq k},  (\d^n(\tilde W_i^{n}, \tilde W_j^{n} )    )_{i,j\leq k} \right  ) \to 0,\]
where $d_{TV}$ stands for the total variation distance. Therefore since $\forall i\in \N$, $\proba(W^n_i=\tilde W^n_i)\to 0$,
\begin{equation} d_{TV} \left (  (\d^n(L_i^{n}, L_j^{n} )  )_{i,j\leq k},  (\d^n(W_i^{n}, W_j^{n} )  )_{i,j\leq k} \right  ) \to 0. \label{1206.5}\end{equation}
And finally, \eqref{1106.233} directly follows from  \eqref{1106.232} and \eqref{1206.5}. This concludes the proof.

\section{Gromov--Hausdorff--Prokhorov convergence and height of $\D$-trees} \label{GHPsection}
\subsection{Preliminaries: technical results on $\tilde \mu$ and $\mu$} \label{7.1}
The aim of this section is to estimate for $\D\in \OmegaD$, the measure $\tilde \mu^\D$ introduced in Algorithm \ref{Alg4}, and  to prove a few important results on an analog $\mu:=\sum_{i=1}^{\n} (d_i-1)\delta_{X_i}$.\begin{lemma} \label{pizza}For every $\D\in \OmegaD$ and $l\in \R^+$,
\[ \proba \left ( \tilde \mu^\D[0,l]\leq \E[ \tilde \mu^\D[0,l]]/2  \right ) \leq e^{-l\E[\tilde \mu^\D[0,l]]/4}.\]
\end{lemma}
\begin{proof} First note that for every $1\leq i \leq \n$, $\proba(\tilde X_i\leq l)=1-e^{-ld_i/\sigma}$, and that $\tilde \mu[0,l]-\E[\tilde \mu[0,l]]$ can be written as a sum of independent centered random variables:
\[ \tilde \mu[0,l]-\E[\tilde \mu[0,l]]=\sum_{i=1}^\n M_i \quad \text{where for $1\leq i\leq \n$} \quad M_i:=\left (\1_{\tilde X_i\leq l}-1+e^{-ld_i/\sigma} \right )(d_i-1)/\sigma. \]

We compute some exponential moments. First for every $i$ such that $d_i\geq 2$,
\begin{align} \E \left [\exp \left (-l \frac{d_i}{d_i-1}M_i \right )\right ] 
& =   \exp \left (-(ld_i/\sigma) e^{-ld_i/\sigma} \right)  \left ( \proba(\tilde X_i\leq l)+\proba( \tilde X_i>l)e^{ld_i/\sigma} \right) \notag \\
& =   \exp \left (-(ld_i/\sigma) e^{-ld_i/\sigma} \right)  \left ( 1-e^{-ld_i/\sigma}+e^{-ld_i/\sigma}e^{ld_i/\sigma} \right) \notag \\
& =  \exp \left (-(ld_i/\sigma) e^{-ld_i/\sigma}+\log \left ( 2-e^{-ld_i/\sigma}\right)\right). \notag
\end{align}
Then since for every $x\geq 0$, $-xe^{-x}+\log(2-e^{-x})\leq x(1-e^{-x})/4$,
\[  \E \left [\exp \left (-l \frac{d_i}{d_i-1}M_i \right )\right ]  \leq \exp \left ((ld_i/\sigma)  (1-e^{-ld_i/\sigma} )/4 \right). \]
So by concavity of $t\to t^{(d_i-1)/d_i}$ on $\R^+$, 
\[  \E \left [\exp \left (-l M_i \right )\right ]  \leq \exp \left (l\frac{d_i-1}{\sigma} (1-e^{-ld_i/\sigma})/4\right). \] 
Therefore, since $M_i=0$ when $d_i\leq 1$, multiplying over all $i$ such that $d_i\geq 2$ we get, 
\[  \E \left [\exp \left ( -l\left (\tilde \mu[0,l] -\E \left [ \tilde \mu[0,l] \right ] \right ) \right ) \right ]\leq \exp \left ( l\E \left [ \tilde \mu[0,l] \right ]/4 \right ).\]
Finally the desired result follows from a simple application of Markov's inequality.
\end{proof}

We now upper bound some "numbers of cuts". More precisely, for every $m\geq 0$ let
\begin{equation} \X_m:=\inf\{l\in \R^+, \E[ \tilde \mu[0,l]] \leq m\}  \quad \text{and let} \quad \xi_m:=\min \{i\in \N, \tilde \mu[0,\tilde Y_i]> m\}. \label{1206.13}\end{equation}
It is easy to check that $l\to \E[ \tilde \mu[0,m]]$ is a continuous and strictly increasing function of $m$ so that for every $0\leq m \leq (\n_0-1)/\sigma$, $\E[\tilde \mu[0,\X_m]]= m$. Moreover, we have the next upper bound on $\xi_m$:

\begin{lemma} \label{cut} For every $0\leq m \leq (\n_0-1)/(2\sigma)$,
\[ \proba \left ( \xi_m\geq 3\X_{2m}m+1\right ) \leq 2e^{-\X_{2m}m/2}. \]
\end{lemma}
\begin{proof} First by Lemma \ref{pizza} $\tilde \mu[0,\X_{2m}] \geq m$ with probability at least $1-e^{-\X_{2m}m/2}$. In the following we work conditionally on $\tilde \mu$ and assume that this event holds.

Note that $\xi_m-1=\max\{i,\tilde  \mu[0,\tilde Y_i]\leq m\}\leq \max\{i, \tilde \mu[0,\hat Y_i]\leq m\}$. Also since  conditionally on $\tilde \mu$, $\{\hat Y_i\}_{i\in \N}$ is a Poisson point process with intensity $\tilde \mu[0,l]dl$, the quantity $\max\{i, \tilde \mu[0,\hat Y_i]\leq m\}$ is a Poisson random variable with mean $\int_0^{\alpha} \tilde \mu[0,t]dt$ where $\alpha:=\max\{a\in \R^+, \tilde \mu[0,a]\leq m\}$. Therefore since by our assumption $\alpha \leq \X_{2m}$ and since $\tilde \mu[0,\alpha)\leq m$, $\xi_m-1$ is bounded by a Poisson random variable of mean $\X_{2m}m$. Finally the result follows from some basic concentration inequalities for Poisson random variables (see e.g. \cite{Massart} p.23).
\end{proof} 
\begin{lemma} \label{measuredeath} The following assertions holds:
\begin{compactitem} \item[(a)] For every $0\leq m \leq (\n_0-1)/(16\sigma)$, $\X_{2m}\leq \sigma/8.$
\item[(b)] For every $x\geq 0$, $\E[\tilde \mu[0,x]]\leq x$. Hence, for every $m\geq 0$, $m\leq \X_m$.
\item[(c)] For every $x\geq 1/2$, $\E[\tilde \mu[0,x]]\geq 1/6$.
\end{compactitem}
\end{lemma}
\begin{proof} 
Toward (a), simply note that $x\mapsto \E[\tilde \mu[0,x]]$ is increasing and that
\[ \E\left [\tilde \mu \left [0,\frac{\sigma}{8} \right ] \right ] =\sum_{i=1}^\n \frac{d_i-1}{\sigma} \left (1-e^{-\frac{d_i}{8}}\right ) \geq \left (1-e^{-\frac{2}{8}}\right )\sum_{i=1}^\n \frac{d_i-1}{\sigma} \1_{d_i\geq 2} = \left (1-e^{-\frac{2}{8}}\right )\frac{\n_0-1}{\sigma}\geq 2m.\]

Similarly for (b), note that 
\[ \E[\tilde \mu[0,x]]=\sum_{i=1}^\n \frac{d_i-1}{\sigma} \left (1-e^{-d_i\frac{x}{\sigma}}\right ) \leq \sum_{i=1}^\n \frac{(d_i-1)(d_i)x}{\sigma^2}=x. \]

Finally for (c), since for every $0\leq x\leq 1$, $1-e^{-x}\geq x/3$, note that 
\[ \E[\tilde \mu[0,1/2]]=\sum_{i=1}^\n \frac{d_i-1}{\sigma} \left (1-e^{-\frac{d_i}{2\sigma}}\right ) \geq \sum_{i=1}^\n \frac{(d_i-1)(d_i)}{6\sigma^2}=\frac{1}{6}. \qedhere \]
\end{proof}
\begin{lemma} \label{prelimGH} Let $(\Dn)_{n\in\N}\in \Omega_\D^\N$ and $\Theta\in \OmegaT$. We have the following assertions:
\begin{compactitem} 
\item[ (a)] If $\Dn\ply \Theta$, then for every $x\in \R^+$, $\psi^n(x)=x\E[\tilde \mu^n[0,x]]\limit x\E[\mu^\Theta[0,x]]=\psi^\Theta(x)$.
\item[ (b)] If $\Dn\ply \Theta$ and Assumption \ref{Hypo3} is satisfied, then $\mu[0,\infty]=\infty$ and the $\Theta$-ICRT is a.s. compact.
\end{compactitem}
\end{lemma}
\begin{proof} Note that (b) directly follows from (a). Indeed, using $\sigma^n\to \infty$, Assumption \ref{Hypo3}, (a), the Fatou's lemma applied to $(\psi^n)_{n\in \N}$ gives $\int_1^\infty \frac{1}{\psi^\Theta}<\infty$. So $\mu[0,\infty]=\infty$, and \cite[Theorem 3.3]{ICRT1} yields the a.s. compactness. 

Toward (a), we have have by bounded convergence, provided some justifications,
\begin{align} \E[\tilde \mu^n[0,x]] & = x+\sum_{i=1}^\n \frac{d^n_i-1}{\sigma^n} \left (1-e^{-xd^n_i/\sigma^n}-x\frac{d^n_i}{\sigma^n}\right ), \notag
\\ & \to x+\sum_{i=1}^\infty \theta_i(1-e^{-x\theta_i}-x\theta_i)=\E[\mu^\Theta[0,x]]. \label{2107} \end{align}
So it remains to justify \eqref{2107}. We have for every $n\in \N$ and $i\in \N$ such that $d_i^n\geq 2$,  
\[ \left (\frac{d_i^n}{\sigma^n} \right)^2\leq 2\frac{d_i^n(d_i^n-1)}{(\sigma^n)^2} \leq 2 \frac{\frac{1}{i}\sum_{j=1}^i d_j^n(d_j^n-1)}{(\sigma^n)^2} \leq \frac{2}{i}.\]
So for every $n\in \N$ and $i\in \N$, since for every $0\leq y\leq \sqrt{2}$, $e^y-1-y\leq y^2$,
\[\frac{d^n_i-1}{\sigma^n} \left (1-e^{-xd^n_i/\sigma^n}-x\frac{d^n_i}{\sigma^n}\right ) \leq  \frac{d_i^n-1}{\sigma^n} \left (\frac {d_i^n}{\sigma^n}\right )^2 \leq \frac{2^{3/2}}{i^{3/2}}.\]
And the convergence in \eqref{2107} follows.
\end{proof}
We now prove introduce important results on $\mu$ to upper-bound the length of the branches:
\begin{lemma} \label{lem:murate}Let $\mu^\D:=\sum_{i=1}^{\n} (d_i-1)\delta_{X_i}$. We have for every $1\leq i < \n-1, k\in \N$, 
\[ \proba\left (Y_{k}=i+1 \middle |A_1,\dots, A_i, Y_k\leq i, Y_{k+1}>i \right )= \frac{\mu[0,i]-k}{\n-1-i}. \]\end{lemma}
\begin{remark} Note that if $Y_k\leq i$, $Y_{k+1}>i$ then $k$ must be the number of repetitions in $A_1,\dots, A_i$. 
\end{remark}
\begin{proof} First, recall that the uniform $\D$-tuple $(A_1,A_2,\dots, A_{\n-1})$ may be constructed by reading the first coordinate of a uniform permutation of $\{(i,j)\}_{1\leq i \leq \n, 1\leq j \leq d_j}$, which has cardinal $\n-1$.  Also $Y_k=i+1$ is equivalent to $A_{i+1}\in \{A_1,\dots, A_i\}$. So, splitting according to $A_{i+1}$,
\[ \proba\left (Y_{k}=i+1 \middle |A_1,\dots, A_i, Y_k\leq i, Y_{k+1}>i \right )=\frac{\sum_{a=1}^\n \1_{x\in A_1,\dots, A_i}(d_x-\#\{1\leq j\leq i, A_j=V_x\})}{\n-1-i}.\]
The numerator can then be rewritten as $\mu[0,i]$ minus the number of repetitions in $A_1,\dots,A_i$.
\end{proof}
Recall the definitions of $(E_i^\D)_{1\leq i <\n^\D}$ and $f^\D$  from above  Lemma \ref{timechange}. By Lemma \ref{timechange}, we can and  will assume by Skorokhod representation theorem that $f^\D(\TSB^\D)=\tilde \TSB^\D$. In that case we have:
\begin{lemma} \label{SummaryF} If $f^\D(\TSB^\D)=\tilde \TSB^\D$ then (a) for every $i\leq \n-1$, $\mu[0,i]= \tilde \mu [0,f(i)]/\sigma$. In particular, (b) for every $0\leq m \leq N/\sigma$, we have $ \xi_m=\min \{i\in \N, \mu[0,Y_i]> \sigma m\}$. 
\end{lemma} 
\begin{proof} (a) is \eqref{2711}. (b) is then direct from the definition of $\xi(m)$ since $\tilde Y_i=f(Y_i)$. 
\end{proof}

\subsection{Height of $\D$-trees: Proof of Theorem \ref{thm:D_Height}}
Fix $\D\in \OmegaD$ and $x\in \R^+$. Recall from \eqref{1206.13} the definitions of $\X_m$ and $\xi_m$. By Lemma \ref{timechange} we may and will assume for the rest of the section that a.s. $f^\D(\TSB^\D)=\tilde \TSB^\D$, so that we may use Lemma \ref{SummaryF}. Also recall that $T^\D_{a}$ is the tree obtained by stopping Algorithm \ref{D-tree} after $a$ steps. In particular, $T^\D_{Y_i}$ is the subtree spawned by the root and the $i$ first leaves. Our proof of Theorem \ref{thm:D_Height} is based on 
\begin{equation} H(T)\leq H(T_{Y_{j_0}})+\sum_{j=0}^{K-1} d_H(T_{Y_{j_i}},T_{Y_{j_{i+1}}})+d_H(T_{Y_{j_K}},T), \label{1206.12} \end{equation}
where $d_H$ denote the Hausdorff distance (see Appendix \ref{GH}) and $0\leq j_0 \leq \dots \leq j_K \leq N$ and $K\in \N$ are well  chosen. We will detail this  choice later.

First, let us introduce some notations. For every let $i\leq j\in \N$, let
\begin{equation} \Delta(i,j):=d_H(T_{Y_i},T_{Y_j}). \label{1406.172} \end{equation} 
Then for every $i\in \N$, let:
\[ \Delta(0,i):=H(T_{Y_i}) \,\, ; \,\, \Delta(i,2^\geq ):=\max_{1\leq a \leq \n, d_i\geq 2} \d(T_{Y_i},V_a) \,\, ; \,\, \Delta(2^\geq ,\infty):=\max_{1\leq b \leq \n} \d(\{V_a\}_{1\leq a\leq\n, d_a\geq 2},V_b). \]
So \eqref{1206.12} can be rewritten with those notations, and by splitting the last term in two,  as:
\begin{equation} H(T)\leq \Delta(0,j_0)+\sum_{j=0}^{K-1} \Delta(j_i,j_{i+1})+\Delta(j_K,2^\geq)+\Delta(2^\geq, \infty). \label{1406} \end{equation}

The rest of the section is organized as follows: We upper bound each of the four terms of \eqref{1406} then we sum the upper bounds to prove Theorem \ref{thm:D_Height}. We first upper bound $\Delta(0,i)\leq Y_i$:
\begin{lemma} \label{2step0} For every $0\leq m \leq (\n_0-1)/(2\sigma)$ such that $\X_{2m}\leq \sigma/8$,
\[ \proba\left (Y_{\xi_m}> 3\frac{\n}{\sigma}\X_{2m}+2 \right ) \leq 5e^{-m\X_{2m}/2}.\] 
\end{lemma}
\begin{proof} Let $\alpha:=\inf\{i\in \N,\mu[0,i]> \sigma m\}$. Let us upper bound $\alpha$ and then $Y_{\xi_m}-\alpha$. By Lemma \ref{SummaryF} (a), $\alpha=\inf\{i\in \N,\tilde \mu[0,f(i)]> m\}$. So by monotony of $x\mapsto \tilde \mu[0,x]$ and by Lemma \ref{pizza},
\begin{equation} \proba(\X_{2m}\leq f(\alpha-1))\leq \proba(\tilde \mu[0,\X_{2m}]\leq m ) \leq e^{-m\X_{2m}/2}. \label{1406.15} \end{equation}
Then, since  $f$ is increasing by definition, and since for every $1\leq i \leq \n$, $f(i)=\sum_{k=1}^i E^\D_k$,
\begin{align} \proba\left ((f(\alpha-1)< \X_{2m}) \cap \left (\alpha\geq 2(\n/\sigma)\X_{2m}+2 \right)   \right )  & \leq  \proba\left (f\left ( 2 (\n/\sigma)\X_{2m}+1 \right )< \X_{2m} \right ) \notag \\ & \leq  \proba \Bigg ( \sum_{i=1}^{ \lceil 2(\n/\sigma) \X_{2m} \rceil } E_i< \X_{2m}  \Bigg). \label{1406.16} \end{align}
Furthermore, since $(E_i)$ is a family of independent exponential random variables of mean greater than $\sigma/\n$, we have by classical results on the Gamma distributions (see e.g. \cite{Massart} Section 2.4),
\[ \proba \Bigg ( \sum_{i=1}^{\lceil 2(\n/\sigma)\X_{2m} \rceil} E_i< \X_{2m}  \Bigg ) \leq e^{-\left ( \n/(2\sigma)\X_{2m} \right )}. \]
Therefore, since $m\leq N/(8\sigma)\leq \n/(8\sigma)$ we have  by \eqref{1406.15} and \eqref{1406.16},
\begin{equation} \proba \left (\alpha\geq 2(\n/\sigma)\X_{2m}+2 \right ) \leq 2e^{-m\X_{2m}/2}.  \label{1406.162} \end{equation}

Now let us upper bound $Y_{\xi_m}-\alpha$. Note from Lemma \ref{SummaryF} (b) that $Y_{\xi_m}$ is the first index of repetition in $\{A_i\}_{1\leq i\leq \n-1}$ after $\alpha$. Note also that $\alpha$ is a stopping time for $\{A_i\}_{1\leq i\leq \n-1}$. So by Lemma \ref{lem:murate}, conditionally on $\{A_{i}\}_{1\leq i \leq \alpha}$, $Y_{\xi_m}-\alpha$ is bounded by a geometric random variable of parameter $(\mu[0,\alpha]-\xi_m+1)/\n$. Hence, 
\begin{equation} \proba\left (\left . Y_{\xi_m}-\alpha\geq (\n/\sigma)\X_{2m}  \right |\{A_{i}\}_{1\leq i \leq \alpha} \right )\leq \exp \left (-\X_{2m}(\mu[0,\alpha]-\xi_m+1)/\s\right). \label{1406.17} \end{equation}
Furthermore by definition of $\alpha$, by $\X_{2m}\leq \sigma/8$, and by Lemma \ref{cut},
\[ \proba\left (\mu[0,\alpha]-\xi_m+1 \leq \sigma m/2 \right ) \leq \proba\left (\xi_m \geq \sigma m/2+1 \right ) \leq \proba\left (\xi_m \geq 3m\X_m+1 \right ) \leq 2e^{-m\X_{2m}/2}. \]
So by \eqref{1406.17},
\[\proba\left ( Y_{\xi_m}-\alpha\geq (\n/\sigma)\X_{2m} \right )\leq \proba \left (\mu[0,\alpha]-\xi_m-1\leq \sigma m/2\right )+e^{-m\X_{2m}/2} \leq 3e^{-m\X_{2m}/2}. \]
Finally summing the above equation with \eqref{1406.162} yields the desired inequality.
 \end{proof}
 
 \begin{lemma} \label{2step} We have the following upper bounds on $\Delta(j,k)$:
 \begin{compactitem}
\item[(a)] For every $1\leq j < k \leq \n_0-1$, and $t>0$, 
\[ \proba \left ( \left . \Delta(j,k) \geq t  \right |\{A_i\}_{1\leq i\leq Y_j} \right )  \leq (k-j) e^{-t(\mu[0,j]-j)/\n}.\]
\item[(b)] For every $0\leq m \leq N/(2\sigma)$ such that $\X_{2m}\leq \sigma/8$ and $t>0$,
\[ \proba \left (\Delta(\xi_{m/2},\xi_m)> t \right )  \leq 2e^{-\X_{2m}m/2}+ 3\X_{2m}m e^{-t(\sigma m)/(8\n)}.\]
\end{compactitem}
\end{lemma}
\begin{proof} Toward (a), note that $\Delta(j,k)=\max_{j<i\leq k}d(L_i,T_{Y_j})$ where $(L_i)_{1\leq i \leq \n_0}$ are the different leaves that are used in Algorithm \ref{D-tree}. Also by symmetry of the leaves, for every $j<i\leq k$
\begin{equation} d(L_i,\T_{Y_j})=^{(d)} d(L_{i+1},\T_{Y_j})=Y_{j+1}-Y_j. \label{1808a} \end{equation}
Hence, for every $t\in \R^+$
\[ \proba \left ( \left . \Delta(j,k)\geq t  \right |\{A_i\}_{1\leq i\leq Y_j} \right )  \leq (k-j)  \proba(Y_{j+1}-Y_j\geq t).\]
On the other hand, by  Lemma \ref{lem:murate}, $Y_{j+1}-Y_j$ is bounded by a geometric random variable of parameter $(\mu[0,j]-j)/\n$ and (a) follows.

Toward (b),
 note that $\Delta(\xi_{m/2},\xi_m)=0$ if $\xi_{m/2}=\xi_m$. So, by union bound, and by Lemma \ref{cut},
\begin{align*} \proba \left (\Delta(\xi_{m/2},\xi_m) > t \right ) & \leq \proba \left (\xi_m\geq 3\X_{2m}m+1 \right ) + \proba \left ((\Delta(\xi_{m/2},3\X_{2m}m+1)> t)\cap (\xi_m< 3\X_{2m}m+1 ) \right )
\\ & \leq 2e^{-\X_{2m}m/2} + \proba \left ((\Delta(\xi_{m/2},3\X_{2m}m+1)> t)\cap (\xi_{m/2}\leq 3\X_{2m}m ) \right ). \end{align*}
Furthermore by (a), since $\xi_{m/2}$ is a stopping time for $(A_i)_{1\leq i \leq \s}$,
\[ \proba \left ((\Delta(\xi_{m/2},3\X_{2m}m+1)> t)\cap (\xi_{m/2}\leq 3\X_{2m}m) \right ) \leq 3\X_{2m}m \E\left [ \exp \left (-t\frac{\mu[0,Y_{\xi_{m/2}}]-3\X_{2m}m}{\n} \right ) \right ]. \]
So, since by Lemma \ref{SummaryF} (b) a.s. $\mu[0,Y_{\xi_{m/2}}]\geq \sigma m/2$, and since $\X_{2m}\leq \sigma/8$,
\[ \proba \left ((\Delta(\xi_{m/2},3\X_{2m}m+1)> t)\cap (\xi_{m/2}\leq 3\X_{2m}m) \right ) \leq 3\X_{2m}m e^{-t(\sigma m)/(8\n)}.\]
This concludes the proof.
\end{proof}

\begin{lemma} \label{2step3} We have the following upper bounds on $\Delta(j,2^\geq)$:
\begin{compactitem}
\item[(a)] For every $1\leq j\leq \n_0-1$ and $t>0$, 
\[ \proba \left ( \Delta(j,2^\geq )> t \right ) \leq s_{\geq 2} e^{-t(\mu[0,Y_j]-j)/\n}.\]
\item[(b)] For every $0\leq m \leq (\n_0-1)/(2\sigma)$ such that $\X_{2m}\leq \sigma/8$ and $t>0$,
\[ \proba \left ( \Delta(\xi_{m},2^\geq )> t \right ) \leq 2e^{-\X_{2m}m/2}+\n_{\geq 2}e^{-t(\sigma m)/(8\n)}.\]
\end{compactitem}
\end{lemma}
\begin{proof}
Toward (a), it is enough to prove that for every $a$ such that $d_a\geq 2$ and $1\leq j\leq N$, 
\begin{equation} \proba \left (\left . d\left (V_a,T_{Y_j}  \right)> t  \right | \{A_i\}_{1\leq i\leq Y_j} \right ) \leq e^{-t(\mu[0,Y_j]-j)/\n}.\label{1709}\end{equation}
To this end, let us use Algorithm \ref{Spann}. Let $W=(W_i)_{1\leq i \leq \n}$ be any permutation of $\V^\D$ such that for every $1\leq i \leq j$, $W_i=L_i$ and such that $W_{j+1}=V_a$. Note that Algorithm \ref{D-tree} and Algorithm \ref{Spann} follow the exact same steps until $Y_{j}$ so a.s. $T_{Y_j}=T^W_{Y_j}$. Thus since $T$ and $T^W$ are both $\D$-trees, 
\[ d (V_a,T_{Y_j})=^{(d)}d^W(V_a,T^W_{Y_j}). \]

Also note that 
\begin{equation} d^W(V_a,T^W_{Y_j})\leq (Y_{j+1}-Y_j) \label{1808b} \end{equation}
 since one of the two following cases must happen:
\begin{compactitem}
\item Either $V_a\notin (A_i)_{1\leq i \leq Y_{j+1}}$ and then Algorithm \ref{D-tree} and Algorithm \ref{Spann} follow the exact same steps until $Y_{j+1}-1$. And at the next step $L_{j+1}$ is "relabelled" $V_a$. So $d^W(V_a,T^W_{Y_j})=(Y_{j+1}-Y_j)$.
\item Or $V_a\in (A_i)_{1\leq i \leq Y_{j+1}}$, and then Algorithm \ref{D-tree} and Algorithm \ref{Spann} follow the exact same steps until $i:=\inf\{k\in \N, A_k=V_a\}\leq Y_{j+1}$. So $d^W(V_a,T^W_{Y_j})=\max(i-Y_j,0)\leq Y_{j+1}-Y_j$.
\end{compactitem}
Therefore \eqref{1808b}  holds, and \eqref{1709} follows from Lemma \ref{2step}.

The proof of (b) is similar to the proof of Lemma \ref{2step} (b). We omit the details.
\end{proof}
\begin{lemma} If $\n_0\geq 2$, then for every $t\geq 0$,
 \label{2step4} 
\[ \proba\left ( \Delta(2^\geq,\infty)>3+t\frac{\n}{\sigma}+\frac{\ln \n_0}{\ln \left ((\n-1)/\n_1 \right)} \right ) \leq 2 e^{-t\n_0/\s}.\]
\end{lemma}
\begin{proof} First, note that $ \Delta(2^\geq,\infty)$ is bounded by the length of the largest path whose vertices have degree 0 or 1. Furthermore, since those paths cannot contain any glue point,
\[ \Delta(2^\geq,\infty) \leq 2+\max\{(j-i), 1\leq i \leq j \leq \n-1, (V_{A_{i+1}},V_{A_{i+2}},\dots, V_{A_j})\subset \mathcal{V}_{1}\}.\]
where $\mathcal{V}_{1}$ is the set of vertices with $d_i=1$.
Then by a simple union bound, for every $t\in \N$,
\begin{align*} \proba\left (\Delta(2^\geq,\infty)\geq 2+t \right ) & \leq \sum_{i=1}^{\n-t-1} \proba \left (V_{A_i}\notin \mathcal{V}_{1}, V_{i+1},\dots, V_{i+t}\in \mathcal{V}_{1}\right ) 
\\ & = \sum_{i=1}^{\n-t-1} \frac{\n-1-\n_{1}}{\n-1}\frac{\n_1}{\n-2}\frac{\n_1-1}{\n-3} \dots \frac{\n_1-t+1}{\n-t-1} 
\\ & =  (\n-1-\n_{1})\frac{\n_1}{\n-1}\frac{\n_1-1}{\n-2} \dots \frac{\n_1-t+1}{\n-t}\leq 2\n_0\left (\frac{\n_1}{\n-1} \right)^{t},
\end{align*}
where we use for the last inequality the fact that $(\n-\n_{1})=\n_0+\n_{\geq 2} \leq 2\n_0$ and $\n_1\leq \n-1$. Thus, by monotony of $t\mapsto \proba(\Delta(2^\geq, \infty)\geq t)$, we have for every $t>0$,
\[ \proba\left ( \Delta(2^\geq,\infty)>3+t \right ) \leq 2\n_0 \left (\frac{\n_1}{\n-1} \right)^t.\]

Therefore,
\[P:= \proba\left ( \Delta(2^\geq,\infty) > 3+t\frac{\n}{\sigma}+\frac{\ln \n_0}{\ln \left ((\n-1)/\n_1 \right)} \right ) \leq 2\left (\frac{\n_1}{\n-1} \right )^{t\n/\sigma}. \] 
Finally, since for every $x,y\geq 0$, $x^y\leq e^{-(1-x)y}$,
\[ P\leq 2 e^{-t(\n/\sigma)(1-\n_1/(\n-1) )} \leq  2 e^{-t(\n/\sigma) (\n_0/(\n-1))}\leq 2 e^{-t\n_0/\sigma}. \qedhere\]
\end{proof} 
 \begin{proof}[\textbf{Proof of Theorem \ref{thm:D_Height}}] Fix $x\geq 0$. Recall that we want to use \eqref{1406} to upper bound $H(T)$.

 Beforehand, let us make some assumptions to exclude some trivial but annoying  cases: Since $C$ can be chosen arbitrary large in Theorem \ref{thm:D_Height}, we may assume, without loss of generality, that $\psi(x)\geq 200$. This assumption together with Lemma \ref{measuredeath} (b) imply that $x\geq 10$. Similarly, we may assume $\n\geq 10$. Also we may assume $\n_0\geq 2$, since otherwise $\sigma=0$ and the result is trivial. Finally we may assume $x\leq \sigma/8$, since otherwise the result is trivial with $c=1/8$ because $H(T)\leq \n$.

We now treat the general case. Let $m:=\frac{1}{2}\E[\tilde \mu [0,x]]$ then let $K:=\inf\{k\in \N, 2^{k+1} m> \frac{\n_0-1}{32\sigma}\}$. To simplify the forthcoming computations, let for every $i\in \N$, $\alpha_i:=2^i m \X_{2^{i+1}m}$.  Note that $\X_{2m}=x\leq \sigma/8$ and that for every $1\leq i \leq K$, $2^im\leq \frac{\n_0-1}{32\sigma}$ so  by Lemma \ref{measuredeath} (a) $\X_{2^{i+1} m}\leq \sigma/8$. Hence, we have by Lemmas \ref{2step0}, \ref{2step} (b), \ref{2step3} (b), \ref{2step4} respectively, for every $0< i \leq K$:
\begin{equation} \proba\left (\Delta(0,\xi_{m})> 3\frac{\n}{\sigma}x+2 \right ) \leq 5e^{-\alpha_0/2}, \label{aaaaa} \end{equation}
\begin{equation} \proba \left (\Delta(\xi_{2^{i-1}m},\xi_{2^i m})>8\frac{\n}{\sigma}\left (\frac{\ln(3\alpha_i)}{2^i m}  +\frac{x}{2^{i/2}} \right )\right )\leq 2e^{-\alpha_i/2}+ e^{-2^{i/2} xm}, \label{bbbbb} \end{equation}
\begin{equation} \proba \left (\Delta(\xi_{2^Km},2^\geq)> 8\frac{\n}{\sigma}\left (\frac{\ln(\n_{\geq 2})}{2^K m}  +\frac{x}{2^{K/2}} \right ) \right ) \leq 2e^{-\alpha_K/2}+e^{-2^{K/2} xm},\label{ccccc} \end{equation}
\begin{equation} \proba\left ( \Delta(2^\geq, \infty) >3+\frac{\n}{\s}x+\frac{\ln \n_0}{\ln \left ((\n-1)/\n_1 \right)}  \right ) \leq 2e^{-x\n_0/\sigma}. \label{ddddd} \end{equation} 
Now let 
\[ A:=3\frac{\n}{\sigma}x+2+8\frac{\n}{\sigma}\sum_{i=1}^{K}\left (\frac{\ln(3\alpha_i)}{2^i m}  +\frac{x}{2^{i/2}} \right ) +8\frac{\n}{\sigma}\left (\frac{\ln(\n_{\geq 2})}{2^K m}  +\frac{x}{2^{K/2}} \right )+3+\frac{\n}{\s}x+\frac{\ln \n_0}{\ln \left ((\n-1)/\n_1 \right)},   \]
and let 
\[ B:=5e^{-\alpha_0/2}+2\sum_{i=1}^{K} \left ( 2e^{-\alpha_i/2}+ e^{-2^{i/2} xm}\right ) +e^{-x \n_0/\sigma}. \]
So that by \eqref{1406}, and \eqref{aaaaa}+\eqref{bbbbb}+\eqref{ccccc}+\eqref{ddddd},
\begin{equation} \proba(H(T)\geq A)\leq B, \label{1506.24}\end{equation}
and it only remains to upper bound $A$ and $B$.

Toward upper bounding $B$, note that $x\mapsto \E[\tilde \mu[0,x]]$ is increasing, so $x\mapsto \X_x$ is increasing, and thus for every $0\leq i \leq K-1$, 
\[ \alpha_i=2^i m\X_{2^{i+1}m} \geq 2^i (m\X_{2m}) = 2^i \alpha_0. \]
Then by definition of $m$, $xm=\X_{2m} m =\alpha_0$, and by assumption, $\alpha_0=x\E[\tilde \mu[0,x]]/2\geq 100$. Hence, by standard comparisons with geometric sums,
\begin{equation*} \sum_{i=1}^{K} \left ( 2e^{-\alpha_i/2}+ e^{-2^{i/2} xm}\right ) \leq \sum_{i=1}^{K}\frac{1}{2^i} \left ( 2e^{-\alpha_0/2}+ e^{-\alpha_0}\right ) \leq 40 e^{-\alpha_0}.  \end{equation*} 
Furthermore note that, $\alpha_0=x\E[\tilde \mu[0,x]]/2\leq x\E[\tilde \mu[0,\infty]]/2=xN/(2\sigma)$. So $e^{-xN/\sigma}\leq e^{-\alpha_0}$ and
\begin{equation} B= 5e^{-\alpha_0/2}+2\sum_{i=1}^{K} \left ( 2e^{-\alpha_i/2}+ e^{-2^{i/2} xm}\right ) +e^{-xN/\sigma} \leq 100 e^{-\alpha_0/2}. \label{BBBBB} \end{equation}

We now upper bound $A$. First, some straightforward inequalities using $x\geq 10$, $\n\leq \s$, $\n\geq \n_1+2$, $2^{K+1} m> (\n_0-1)/(32\sigma)$ and $\n\geq 10$ gives: 
\begin{equation} A\leq 800\frac{\n}{\s} \left (x+\frac{\s}{\n}\frac{\ln \n_0}{\ln \left (\n/\n_1 \right)}+\frac{\s}{\n_0}\ln(\n_{\geq 2})\right )+ 8\frac{\n}{\sigma}\sum_{i=1}^{K}\frac{\ln(3\alpha_i)}{2^i m}. \label{AAAA} \end{equation}
Then note that for every $l\geq 0$, $\psi(l)\leq l\sum_{i=1}^\n (d_i-1)/\sigma\leq l \n_0/\sigma$. Thus, 
\begin{equation} \int_1^\sigma \frac{dl}{\psi(l)}\geq \log(\sigma) \frac{\sigma}{\n_0} \geq \frac{\s}{\n_0}\ln(\n_{\geq 2}). \label{ShitA} \end{equation}
Similarly, using for $x\geq 0$, $1/\ln(1+x)\leq 2+2/x$, and $\n\geq \n_1+\n_0$, we have
\begin{equation} \frac{\sigma}{\n}\frac{\ln \n_0}{\ln(\n/\n_1)}\leq \frac{\sigma}{\n}\ln(n_0)\left (2+\frac{\n_1}{\n_0} \right ) \leq 3 \log(\sigma) \frac{\sigma}{\n_0} \leq 3 \int_1^\sigma \frac{dl}{\psi(l)}. \label{ShitB} \end{equation}

Thus, it remains to upper bound $\sum_{i=1}^{K} \frac{\ln(3\alpha_i)}{2^i m}$. To this end, we may assume that $K\geq 1$ since otherwise the sum is null. In that case, $2^Km\leq (\n_0-1)/(32\sigma)$. Moreover, we have by Lemma \ref{measuredeath} (b) and since for every $1\leq i \leq \n$, $\alpha_i\geq 50$,  \begin{equation} \sum_{i=0}^{K-1}\frac{\ln(3\alpha_i)}{2^i m}  \leq 2 \sum_{i=1}^{K}\frac{\ln(\alpha_i)}{2^i m}=2 \sum_{i=1}^{K} \frac{\ln(\X_{2^{i+1}m} 2^{i}m)}{2^i m} 
\leq 4 \sum_{i=1}^{K} \frac{\ln(\X_{2^{i+1}m})}{2^i m}. \label{1506.14}
\end{equation}
Next, we compare the last sum with $\int dl/\psi(l)$. By Lemma \ref{psimu}, and standard integral calculus,
\begin{align*} \int_{\X_{2m}}^{\X_{m2^{K+1}}} \frac{dl}{\psi(l)}  =\sum_{i=1}^{K} \int_{\X_{m2^i}}^{\X_{m2^{i+1}}} \frac{dl}{l\E[\tilde \mu[0,l]]} \geq \sum_{i=1}^{K} \int_{\X_{m2^i}}^{\X_{m2^{i+1}}} \frac{dl}{l m2^{i+1}}
 \geq \sum_{i=1}^{K} \frac{\ln(\X_{2^{i+1}m})}{2^{i+2} m}-\frac{\ln(\X_{2m})}{2m}.
 \end{align*}
We then recall that $x=\X_{2m}\geq 10$ and $m\geq 1/6$ (by Lemma \ref{measuredeath} (c)). Also $2^{K+1} m\leq (\n_0-1)/(16\sigma)$ by definition of $K$ so by Lemma \ref{measuredeath} (a), $\X_{2^m 2^{K+1}}\leq \sigma/8\leq \sigma$. Thus, 
\begin{equation} \int_{x}^\sigma \frac{dl}{\psi(l)}+ x\geq \sum_{i=1}^{K} \frac{\ln(\X_{2^{i+1}m})}{2^{i+2} m}. \label{ShitC} \end{equation}
Therefore, plugging \eqref{ShitA}, \eqref{ShitB}, \eqref{ShitC} in \eqref{AAAA} gives for some constant fixed constant $C>0$, 
\begin{equation} A\leq Cx+C\int_{1}^\sigma \frac{dl}{\psi(l)}. \label{AAAAAA} \end{equation} 
To sum up, \eqref{AAAAAA}, \eqref{BBBBB} and \eqref{1506.24} yield together the desired inequality.
\end{proof}
\begin{remark} It is possible to get a slightly better bound, up to some $\log$, if one does not do the steps \eqref{ShitA}, \eqref{ShitB}, and does not use the bound $\X_{2^m 2^{K+1}}\leq \sigma$ for \eqref{ShitC}. However, the resulting bound is so technical it does not seem worth the effort. If one wants to go even beyond that, then one needs to improve Lemma \ref{2step3}, which already only contributes  $\log(\n_{\geq 2})(\sigma/n_0)\leq \log(\n_{\geq 2})$, or find exact constants which is impossible with the methods of this paper.
\end{remark}
 \subsection{GHP convergence of $\D$-trees: Proof of Theorem \ref{D_GHP_T}}
We work under the setting of Theorem \ref{D_GHP_T}: $\Dn \ply \Theta$, $\M^\Dn \to 0$, and Assumption \ref{Hypo3} holds. By Theorem \ref{D_GP_T}, the next convergence holds weakly for the GP topology:
\begin{equation} \left ( T^n,(\s^n/n) \d^n,\M^n \right) \limit (\T^\Theta,\d^\Theta,\p^\Theta). \label{2209c}  \end{equation}
We need to prove that it also also holds for the  GHP topology. To do so, by Lemma \ref{GP+GH=GHP}, and since $\p^\Theta$ have a.s. support $T^\Theta$ (see \cite{ICRT1} Theorem 3.1), it suffices to show that \eqref{2209c} holds for the Gromov--Hausdorff (GH) topology (see Section \ref{GH}). To this end, the main idea is to show that (i) the first branches of $\Dn$-trees converge for the GH topology, and that (ii) $\Dn$-trees are close from their first branches. Toward (i), it is straightforward to check from Proposition \ref{pro:First_branches_D} that  for every $k\in \N$, 
 \begin{equation}\left ( T_{Y_k^n}^\n,(\s^n/n) d^n,\M^n \right) \limit^{\WGH} \left (\T_{Y_k^\Theta}^\Theta,d^\Theta,\p^\Theta \right ). \label{2209e}
 \end{equation}
Indeed, one can construct both $\D$-trees and ICRT in $(\R^+)^{\N}$ using a dimension for each branches (see Aldous \cite{Aldous1}), and since by Proposition \ref{pro:First_branches_D} the cuts and glue points converge, it directly follows that the subtrees obtained from the first branches converge. We omit the straightforward details.
 
Toward (ii), we adapt the proof of Theorem \ref{thm:D_Height} to prove the following tightness result:
\begin{lemma} \label{GHGHGH} For every $\e>0$,
\[ \lim_{k \to \infty}  \limsup_{n\to +\infty} \proba \left (\frac{\s^n}{n}d_H \left (T^n_{Y_k^n},T^n \right )>\e  \right )=0. \]
\end{lemma}
Before proving Lemma \ref{GHGHGH}, let us explain why it implies together with \eqref{2209e}, Theorem \ref{D_GHP_T}.
\begin{proof}[Proof of Theorem \ref{D_GHP_T}.] Fix $\e>0$. Let $k\in \N$. By the Skorokhod's representation theorem, we may assume that the convergence  \eqref{2209e} holds a.s. for the GH topology. 

Then, by Lemma \ref{GHGHGH}, if $k\in \N$ is large enough, we have for every $n\in \N$ large enough,
\[ \proba \left (\frac{\s^n}{n} d_H \left (T^n_{Y_k^n},T^n \right )>\e  \right )\leq \e. \]
Also, since by Lemma \ref{prelimGH} (c) the $\Theta$-ICRT is a.s. compact, if $k\in \N$ is large enough,
\[ \proba \left (d_H \left (\T^\Theta_{Y_k^\Theta },\T^\Theta \right )>\e  \right )\leq \e. \]
Thus, given the a.s. GH convergence in \eqref{2209e}, if $k\in \N$ is large enough, we have for every $n\in \N$ large enough,
\[ \proba \left (d_{\GH} \left (\left ( T^n,(\s^n/n) \d^n,\M^n \right),\left (\T^\Theta,\d^\Theta,\p^\Theta \right ) \right )>2\e  \right )\leq 2\e. \]
Finally, since $\e>0$ is arbitrary, Theorem \ref{D_GHP_T} follows.
\end{proof}
Our goal for the rest of the section is to prove Lemma \ref{GHGHGH}. We keep the same notations as in the proof of Theorem \ref{D_GHP_T}. 
By adapting the proof of Theorem \ref{thm:D_Height} we get the following result.
\begin{lemma} \label{Tight1} There exists $C>0$ such that for every $\D\in \OmegaD$, $10\leq x\leq \sqrt{\sigma}/8$, and $t>0$, if $m=\E[\tilde \mu[0,x]]/2$,
\[ \proba\left (d_H(T_{Y_{\xi_m}},T)>C+C\frac{\n}{\s} \left (t+\int_{x}^{\sigma} \frac{dl}{\psi(l)} \right ) \right ) \leq Ce^{-x\E[\tilde \mu[0,x]]/C}+ C e^{-t\E[\tilde \mu[0,x]]/C}. \]
\end{lemma}
\begin{proof} We keep the same notations as in the proof of Theorem \ref{D_GHP_T}. First note that 
\[ d_H(T_{Y_{\xi_m}},T)=\Delta(\xi_m,\infty)\leq \sum_{i=1}^{K} \Delta(\xi_{2^{i-1}m},\xi_{2^i m}) +\Delta(\xi_{2^Km},2^\geq)+\Delta(2^\geq, \infty),  \]
Then by Lemmas, \ref{2step} (b), \ref{2step3} (b), \ref{2step4} respectively, we have for every $0< i \leq K$,
\begin{equation*} \proba \left (\Delta(\xi_{2^{i-1}m},\xi_{2^i m})>8\frac{\n}{\sigma}\left (\frac{\ln(3\alpha_i)}{2^i m}  +\frac{t}{2^{i/2}} \right )\right )\leq 2e^{-\alpha_i/2}+ e^{-2^{i/2} t m}. \label{bbbbbe} \end{equation*}
\begin{equation*} \proba \left (\Delta(\xi_{2^Km},2^\geq)> 8 \frac{\n}{\sigma}\left (\frac{\ln(\n_{\geq 2})}{2^K m}  +\frac{t}{2^{K/2}} \right ) \right ) \leq 2e^{-\alpha_K/2}+e^{-2^{K/2}  t m}.\label{ccccce} \end{equation*}
\begin{equation*} \proba\left ( \Delta(2^\geq, \infty) >3+\frac{\n}{\s}t+\frac{\ln \n_0}{\ln \left ((\n-1)/\n_1 \right)}  \right ) \leq e^{-t N/\sigma}. \label{ddddde} \end{equation*}
So we can sum the upper bounds and conclude as for Theorem \ref{thm:D_Height}. We omit the details.
\end{proof}
We then deduce Lemma \ref{GHGHGH} from Lemma \ref{Tight1}.
\begin{proof}[Proof of Lemma \ref{GHGHGH}:] Fix $x>0$ large enough. Let $t>0$. Let $m^n:=\E[\tilde \mu^n[0,x]]/2$. By Lemma  \ref{Tight1}, for every $n$ large enough,
\[\proba\left (d_H(T^n_{Y_{\xi_{m^n}}},T^n)>C+C\frac{n}{\s^n} \left (t+\int_{x}^{\sigma^n} \frac{dl}{\psi^n(l)} \right ) \right ) \leq  C e^{-t\E[\tilde \mu^n[0,x]]/C}+Ce^{-x\E[\tilde \mu[0,x]]/C}. \label{1606.11} \]
So by using Lemma \ref{cut}, and Lemma \ref{measuredeath} (b) to bound $\xi_{m^n}$, for every $x,n$ large enough, 
\[\proba\left (d_H(T^n_{Y_{\lceil 3x^2\rceil }},T^n)>C+C\frac{n}{\s^n} \left (t+\int_{x}^{\sigma^n} \frac{dl}{\psi^n(l)} \right ) \right ) \leq   C e^{-t\E[\tilde \mu^n[0,x]]/C}+(C+1)e^{-x\E[\tilde \mu^n[0,x]]/C}. \]
Now since for every $l>0$ $\E[\mu^n[0,l]]\to \E[\mu^\Theta[0,l]]$ by Lemma \ref{prelimGH} (a), and since $\mu^\Theta[0,\infty]=\infty$, the right hand-side converges toward to $0$ as $n\to \infty$, $x\to \infty$. We may furthermore use Assumption \eqref{D_GHP_T} to get $\int_{x}^{\sigma^n} \frac{dl}{\psi^n(l)}\to 0$ as $n,x\to \infty$. So we get by taking $n\to \infty$, and $x\to \infty$ slowly enough, 
\[ \lim_{k\to \infty} \limsup_{n\to \infty} \proba\left (d_H(T^n_{Y_k},T^n)>2C \frac{n}{\s^n} t \right )=0. \]
Since $t>0$ is arbitrary, this concludes the proof. 
\end{proof}

\section{Degenerate limit toward $\P$-trees} \label{sec:Foutoire} \label{sec:Ptree}
In this section we describe the case where $d_i^n/n \nrightarrow 0$. More precisely, instead of assuming $\Dn\ply \Theta$, we assume that for a sequence $\P=(p_i)_{i\in \N\cup\{\infty\}}$ such that
\begin{equation} p_1>0 \quad ; \quad p_1\geq p_2\geq\dots \quad ; \quad \sum_{i\in \N\cup\{\infty\}} p_i=1, \label{eq:Pset} \end{equation}
\begin{Hypo}[$\Dn \ply \P$] \label{Hypo1} For all $i\geq 1$, $d_i^n/n\to p^\P_i$. 
\end{Hypo}
In that case we observe a step by step convergence of Foata--Fuchs construction (Algorithm \ref{D-tree}), which motivates the following extension of Aldous--Camarri--Pitman \cite{IntroICRT2} $\P$-tree model.

 Let $\OmegaP$ be the set of sequence $(p_i)_{i\in \N\cup \{\infty\}}$ in $\R^+$ with \eqref{eq:Pset}. Let $V_\infty$, $(V_{\infty,i})_{i\in \N}$, and $(L_i)_{i\in \N}$ be different vertices from $(V_i)_{i\in \N}$. For every $\P\in \OmegaP$, the $\P$-tree is the tree constructed as follows:
\begin{algorithm} \label{P-tree} \emph{Definition of the $\P$-tree for $\P\in \OmegaP$.}
\begin{compactitem}
\item[-] Let $(A^\P_i)_{i\in \N}$ be a family of i.i.d. random variables such that for all $i\in \N$, $\proba(A^\P_1=V_i)=p_i$.
\item[-] For every $i\in \N$, let $B^\P_i=A_i$ if $A_i\in \N$, and let $B^\P_i=V_{\infty,i}$ otherwise.
\item[-] 
Let $T^\P_1:=(\{B_1\},\emptyset)$ then for every $i\geq 2$ let
\begin{equation*} T^\P_i:=\begin{cases} T_{i-1}\cup \{B_{i-1},B_{i}\} & \text{if } B_{i}\notin T_{i-1}. \\T_{i-1} \cup \{A_{i-1},L_{\inf\{k, L_k\notin T_{i-1}\}}\}& \text{if } B_{i} \in T_{i-1}. 
\end{cases} \label{1002} \end{equation*}
\item[-] Let $T^\P$ denote the rooted tree $(\bigcup_{n\in \N}T_n,B_1)$.
\end{compactitem}
\end{algorithm}
\begin{remark} The introduction of $V_{\infty}$ allows us to consider a slightly more general definition than the one introduced in \cite{IntroICRT2} which requires $p_{\infty}=0$. 
\end{remark}
Again, we may define $X_i^\P:=\inf\{j\in \N, B_j=V_i\}$, let $Y^\P_1, Y^\P_2\dots, $ be the repetitions in $B_1,\dots, $ the for every $i\in \N$, let $Z_i:=\inf\{j\in \N, B_j= V_{Y_i}\}$. Also let $\n^\P:=\max\{i\in \N\cup\{\infty\}, p_i>0\}$. And let $\sigma^\P :=\sum_{i\in \N} p_i^2$. We now prove convergence of the first branches: 
\begin{proposition} \label{pro:First_branches_DP}Assume that $\Dn \ply \P$ then we have the following weak joint convergence:
\[ \TSB^\D=(X^n_i,Y^n_i,Z^n_i)_{i\in \N}\to\TSB^\P:= (X^\P_i,Y^\P_i,Z^\P_i)_{i\in \N}\]
\end{proposition}
\begin{remark} Since, the cuts and glue points describes the distances in the trees, the above convergence implies the joint convergence of the distances between the $V_i,V_j$, between the leaves$\dots$
\end{remark}
\begin{proof} Consider a topology on $\{V_i\}_{i\in \N\cup\{\infty\}}$ such that as $i\to \infty$, $V_i\to V_\infty$. It directly follows from $\D^n\ply \P$ that $(A_i^n)_{1\leq i\leq n-1}$ converges weakly toward $(A_i^\P)_{i\in \N}$. By Skorokhod's representation theorem we may assume this convergence holds almost surely. It directly follows that $X_i^n \to X^\P_i$ for every $i\in\N$. Also each index of repetition for  $(B_i^\P)_{i\in \N}$ also corresponds to a repetition for $(A_i^n)_{1\leq i \leq n-1}$ with $n$ is large enough. Thus it suffices to show that there is no extra repetitions in  $(A_i^n)_{1\leq i \leq n-1}$. More precisely it is enough to prove that for every $a\in \N$ 
\begin{equation} \limsup_{k\in \infty} \lim_{n\to \infty} \proba(\exists 1\leq i < j \leq a, A^n_i=A^n_j\notin \{V_1,\dots, V_k\})=0. \label{13/02/13h} \end{equation}
By an union bound this last probability is bounded by 
\[ \sum_{1\leq i<j\leq a} \sum_{x>k} \proba(A^n_i=A^n_j=V_x)\leq a^2 \proba (A^n_1=A^n_2=V_b)= \sum_{x>k} \frac{d^n_x(d^n_x-1)}{(\n-1)(\n-2)}, \]
then using for every $x>k$, $d^n_x\leq (d^n_1+d^n_2+\dots d^n_k)/k \leq (\n-1)/k$, the last term is bounded by,
\[ \frac{1}{k} \sum_{x>k} \frac{d^n_x-1}{\n-2} \leq \frac{1+o(1)}{k} .\]
Thus, \eqref{13/02/13h} holds which thus gives the convergence of the $Y_i^n$ and so of the $Z_i^n$. 
\end{proof}

\section{Extensions of the main results} \label{sec:randomD} \label{sec:extension}
\subsection{$\P$-trees and ICRT}
We state  here results analog to Theorems \ref{D_GP_T}, \ref{D_GHP_T}, \ref{thm:D_Height}. The proofs are similar up to two main differences, first we used some equality in distributions for $\D$-trees which are directly extended to $\P$-trees and ICRT by taking the limit of $\D$-tress. (The law of the distance between the leaves $L_i, L_j$ does not depends on $i\neq j$, and several similar claims.) Also to prove the convergence of the first branches of $\P$-trees toward ICRT one need to use the following important analog of Proposition \ref{timechange} due to Aldous, Camarri, and Pitman \cite[Section 4.1]{IntroICRT2} (a bit reformulated here for our purpose).
\begin{lemma} \label{Transform} Let $\P\in \OmegaP$, let $\Theta=(0,p_1/\sigma, p_2/\sigma,\dots)$. Note that $\Theta\in \OmegaT$. Let $\{E^\P_i\}_{i\in \N}$ be a family of i.i.d. exponential random variables of mean $\sigma$. Let $f^\P:\R^+\mapsto \R^+$ increasing be such that for every $i\geq 0$, $f^\P(i)=\sum_{k=1}^i E_k^\P$. Then $f^\P(\TSB^\P)$ and $\TSB^\Theta$ have the same distribution.
\end{lemma} 
\begin{remark} We may morally interpret the above lemma as follows: If one gives exponential length to each edges of a $\P$-tree then one recover an ICRT with $\mu^\Theta<\infty$. 
\end{remark}
We won't detail more on the proofs, and now state the analogs of our main results: For $\P\in \OmegaP$, let $\V^\P:=\{i\in \N, p_i>0\}$, and let $\d^\P$ be the graph distance on $T^\P$. Consider $(\Pn)_{n\in \N}\in \OmegaP^\N$. For more convenient notations, we use superscript $n$ instead of $\Pn$. For every $n\in \N$ consider a random probability measure $\p^n$ on $\{1,\dots, \n^\P\}$ such that $\p^n\to 0$ uniformly.  We assume
\begin{Hypo}[$\Pn \ply \Theta$]\label{Hypo2P} $p_1\to 0$. And for all $i\geq 1$, $p_1^n/\sigma^n\to \theta_i$.
\end{Hypo} 
\begin{theorem} \label{P_GP_T} If $\Pn\ply \Theta$, $\p^n\to 0$ uniformly, and \eqref{eq:ThetaSet} holds, then the following convergence holds for the weak Gromov--Prokhorov (GP) topology (see Appendix \ref{GPdef} for definition):
\[ \left ( \V^n,\s^n \d^n,\M^n \right) \limit^{\WGP} (\T^\Theta,d^\Theta,\p^\Theta).  \]
\end{theorem}
This result was already proved by Aldous, Camarri, Pitman \cite{IntroICRT1,IntroICRT2} in the natural case $p_\infty=0$. We then state our result for the GHP topology. For every $l\in \R^+$, $\P=(p_i)_{i\in\N\cup\{\infty\}} \in \OmegaP$ let 
\[ \psi^\P(l):=l\sum_{i\in \N} \frac{p_i}{\sigma^\P}(1-e^{-p_il/\sigma^\P}). \]
 As for $\D$-trees, we assume:
\begin{Hypo} \label{Hypo3} \label{P_Tight_GHP}
\[ \lim_{y\to +\infty} \limsup_{n\to +\infty} \int_{y}^{\sigma^n\n^{n}} \frac{dl}{\psi^{n}(l)} =0. \]
\end{Hypo} 
\begin{theorem} \label{P_GHP_T} Under the same setting as Theorem \ref{P_GP_T}, if furthermore Assumption \ref{P_Tight_GHP} holds, then the following convergence holds for the weak Gromov--Hausdorff--Prokhorov (GHP) topology
\[ \left ( \V^n,\s^n d^n,\M^n \right) \limit^{\WGP} (\T^\Theta,d^\Theta,\p^\Theta).  \]
\end{theorem}
We finish with similar bounds on the height of $\P$-trees and ICRT.
\begin{theorem} \label{thm:P_Height} There exists $c,C>0$ such that for every $\P\in \OmegaP$ and $x\in \R^+$: 
\[ \proba\left (c\s^{\P}H(T^{\P})>x+\int_{1}^{\sigma^\P \n^\P} \frac{dl}{\psi^\P(l)}\right)\leq C e^{-c\psi^\P(x)}. \]
\end{theorem}
\begin{theorem} \label{thm:ICRT_Height} There exists $c,C>0$ such that for every $\Theta\in \OmegaT$ and $x\in \R^+$:
\[ \proba\left (cH(\T^{\Theta})>x+\int_1^{+\infty} \frac{dl}{\psi^\Theta(l)}\right)\leq Ce^{-c\psi^\Theta(x)}. \]
\end{theorem}
\subsection{Random degree sequences} \label{subsec:RandomDegree}
The aim of this section is twofold: explain how our main results on $\D$-trees can be used to study trees with random degree sequence, then prove that L\'evy trees are ICRT with random parameters. 

Beforehand, we may naturally extend Assumptions \ref{Hypo2} and \ref{Hypo1} and our main results to the case where we do not have necessarily $\n^\Dn=n$ by assuming instead $\n^\Dn\to \infty$, and replacing where needed $n$ by $\n^\Dn$. We still conveniently use the notations $\Dn \ply\P$ and $\Dn\ply\Theta$ in those cases.

Then, we define a topology on the degree sequences $\Omega:=\OmegaD\cup \OmegaP\cup \OmegaT$. 
\begin{compactitem}
\item For every $\D\in \OmegaD$ and $i\in \N$, let $\theta_i^\D:=d_i^\D/\sigma^\D$ and let $\lambda^\D:=\n^\D/\sigma^\D$. 
\item For every $\P\in \OmegaP$ and $i\in \N$, let $\theta_i^\P:=p_i^\P/\sigma^\P$, $\n_0^\P:=\infty$ and let $\lambda^\P:=\lambda^\P$. 
\item For every $\Theta\in \OmegaT$, let $\n_0^\Theta:=\infty$ and let $\lambda^\Theta:=\infty$. 
\end{compactitem} For every sequence $\{\Lambda_n\}_{n\in \N}$ in $\Omega$ and $\Lambda\in \Omega$, we say that $\Lambda_n\to^\Omega \Lambda$ if and only if 
\[ \n_0^{\Lambda_n}\to \n_0^\Lambda \quad ; \lambda^{\Lambda_n}\to  \lambda^{\Lambda} \quad ; \quad \forall i\in \N, \, \theta_i^{\Lambda_n}\to  \theta_i^{\Lambda}. \]
\begin{lemma} \label{puretopo} For every $(\Dn)\in \OmegaD^\N$, $(\Pn)\in \OmegaP^\N$, $(\Theta_n)\in \OmegaT^\N$, $\P\in \OmegaP$, $\Theta\in \OmegaT$, we have: (a)
$\Dn\ply \P \iff \Dn\to^\Omega \P,\, \, $ (b) $ \P_n\ply \Theta \iff \P_n\to^\Omega \Theta,\, \,$ and (c)  $\OmegaD$ is dense on $\Omega$.
\end{lemma}
\begin{proof} Toward (a), if $\Dn\to^\Omega \P$ then for every $i\in \N$, $d_i^\Dn/\n^\Dn= \theta_i^\Dn /\lambda^\Dn \to \theta_i^\P/\lambda^\P=p_i$. Also, since $\n_0^\Dn\to \infty$, we have $\n^\Dn\to \infty$. Hence, $Dn\ply \P$.

On the other hand, if $\Dn\ply \P$, then it directly follows from Fubini's Theorem that 
\[ (\lambda^\Dn)^{-2}=(\sigma^\Dn/\n^\Dn)^2= \sum_{i=1}^\infty \frac{(d^\Dn_i)(d^\Dn_i-1)}{(\n^\Dn)^2}\limit \sum_{i=1}^\infty p_i^2=(\sigma^\P)^2=(\lambda^\P)^{-2}. \]
It then follows that for every $i\in \N$, 
\[ \theta_i^\Dn=(d_i^\Dn/\n^\Dn)\lambda^\Dn \limit (p_i^\P)\lambda^\P= \theta_i^\P. \]
Also since $\n^\Dn\to \infty$, and $p_1^\P>0$, we have $d_1^\Dn\to \infty$ and so $\n_0^\Dn\to \infty$. So $\Dn\to^\Omega \P$. The proof of (b) is similar. We omit the details. Toward (c). Let us show that $\OmegaP$ is included in $\bar \Omega_D$, the adherence of $\OmegaD$.  Fix $\P\in \OmegaP$. Let $(\Dn)_{n\in \N}\in \OmegaD^\N$ such that $\n^\Dn\sim n$ and such that for every $i\in \N$ and $n\in \N$ large enough, $d_i=\lfloor p_i n\rfloor$. Note that $\Dn\ply \P$. Hence, since $\P$ is arbitrary $\OmegaP\subset \bar \Omega_D$. Similarly, $\OmegaT\subset \bar \Omega_D$. 
\end{proof}

Next, to consider Gromov--Prokhorov convergence, we define the set of couple of degree sequences and measure $\Omega_{\Lambda,\M}:=\Omega_{\D,\M}\cup \Omega_{\P,\M}\cup \Omega_{\Theta,\M}$ where
\begin{compactitem}
\item $\Omega_{\D,\M}$ is the set of $(\D,\M)$ such that $\D\in \OmegaD$ and $\M$ is a probability measure on $\{1,\dots, \n^\D\}$.
\item  $\Omega_{\P,\M}$ is the set of $(\P,\M)$ such that $\P\in \OmegaP$ and $\M$ is a probability measure on $\{1,\dots, \n^\P\}$.
\item $\Omega_{\Theta,\M}$ is the set of couples $(\Theta,0)$ such that $\Theta\in \OmegaT$ and $\mu^\Theta[0,\infty]=\infty$. 
\end{compactitem}
Let $\Omega_\p$ be the set of probability measure on $\{V_i\}_{i\in \N}$, and note that $\Omega_{\Lambda,\M}\subset \Omega\times \Omega_\M$. So we may equip $\Omega_{\Lambda,\M}$ with the product topology, where we consider on $\Omega_\p$ the uniform convergence. Then,
\begin{compactitem} 
\item For every $(\D,\M) \in \Omega_{\D,\M}$,  let $\proba^{\D,\M}_\GP$ be the distribution of $T^{\D,\p}_\star:=( \V^\D,(\s^\D/\n^\D) d^\D,\M)$. 
\item For every $(\P,\M) \in \Omega_{\P,\M}$,  let $\proba^{\P,\M}_\GP$ be the distribution of $T^{\P,\p}_\star:= ( \{B^\P_i\}_{i\in \N},\s^\P d^\P,\M)$.
\item
 For every $(\Theta,\M)\in \Omega_{\Theta,\M}$, let $\proba^{\Theta,\M}_\GP$ denotes the distribution of $T^\Theta_\star:=(\T^\Theta,d^\Theta,\p^\Theta)$. 
\end{compactitem}
We have the following technical results:
\begin{lemma}\label{THMGP}The map $(\Lambda,\M) \to \proba^{\Lambda,\M}_\GP$ is continuous for the weak GP topology.
\end{lemma}
\begin{remark} Note that the above lemma provides an alternatif proof for Theorem \ref{P_GP_T}.
\end{remark}
\begin{proof} First, since $\OmegaD$ is dense on $\Omega$, it is easy to check that $\Omega_{\D,\M}$ is dense on $\Omega_{\Lambda,\M}$. Thus, since the GP topology is polish, the desired result follows from the GP convergence of $\Dn$-trees toward $\P$-trees (Proposition \ref{pro:First_branches_DP} after reformulation) and of $\Dn$-trees toward ICRT (Theorem \ref{D_GP_T}).
\end{proof} 

\begin{lemma} \label{PolishBS} The spaces $\Omega_\Lambda$ and $\Omega_{\Omega,\p}$ are Polish (separable, complete, metrisable).
\end{lemma}
\begin{proof} It is well known that $\Omega_\p$ is Polish. And it is easy to check from Lemma \ref{puretopo}  (c) that $\Omega_\Lambda$ is also Polish. We omit the details.
\end{proof}

It directly follows from the last two results that: 
\begin{proposition} \label{mix} Let $\Omega^\proba$ be the set of random variables on $(\Omega,\to^\Omega)$, and let $\Omega_{\Lambda,\M}^{\proba}$ be the set of random variables on $\Omega_{\Gamma,\M}$. 
If  $(X^n,\M^n)_{n\in \N}$ is a sequence in $(\Omega_{\Lambda,\M}^\proba)$ which converges weakly toward $(X,\M)\in \Omega^\proba_{\Lambda,\M}$ then $T^{X_n,\p_n}_\star$ converges weakly toward $T^{X,\p}_\star$ for the GP topology. 
\end{proposition} 
We finish this section by proving a conjecture of Aldous, Miermont, and Pitman's \cite{ExcICRT}: 
\begin{theorem} L\'evy trees, that are GP limits of Galton--Watson trees, are equal in distribution for the GP topology to ICRT with a random parameter $\Theta$. 
\end{theorem}
\begin{proof} For concision, we will not detail the definitions of Galton--Watson trees and L\'evy trees, and only recall the properties we need for the proof. Most importantly, Galton--Watson trees conditioned to their degree sequence are uniform. In other words the fact that the L\'evy tree $T_\star^L=(T^L,\d^L,\p^L)$ is a GP limit of Galton--Watson trees can be expressed as follows: there exists a sequence of random variables $(\D^n,\p^n)$ in $\Omega_{\D,\p}$ such that we have the following weak convergence for the Gromov--Prokhorov topology 
\[ T_\ast^{\D^n,\p^n} \limit T_\star^L. \] 

To show the desired result it is enough to show that there exists a subsequence $(m_n)_{n\in \N}$ and a random $(\Theta,\M)\in \Omega_{\Theta,\M}$ such that we also have the following  convergence for the Gromov--Prokhorov topology
\[ T_\ast^{\D^{m_n},\p^{m_n}} \limit T_\star^{\Theta,0}. \]
Indeed, the unicity of the limit would then imply $T_\star^L=^{(d)} T_\star^{\Theta,0}=(T^\Theta,\d^\Theta,\p^\Theta)$.
By continuity of $(\Lambda,\M) \to \proba^{\Lambda,\M}_\GP$ (see Lemma \ref{THMGP}), and since we work on Polish spaces (see Lemma \ref{PolishBS}) it suffices to show get the following weak convergence on $\Omega_{\Lambda,\p}$
\begin{equation} (\D^{m_n},\p^{m_n})\limit (\Theta,0). \label{eq:ExtractedCV} \end{equation}

Our argument is general, but requires some very basic properties of L\'evy trees, (see Duquesne, Le Gall \cite{Duquesne2,Duquesne1}) to show that we must have in probability 
\[ \text{(a)} \quad \n_{0}^{\Dn}\to \infty  \quad  \text{(b)} \quad \lambda^{\Dn}\to \infty \quad  \text{(c)} \quad  \max_{1\leq i\leq \n^{\Dn}}  \p^{n}(i)\to 0. \]
Toward (a), since L\'evy trees a.s. have an infinite number of branchpoints, $\n^{\Dn}_0\geq \n^{\Dn}_{\geq 2}\to \infty$ in probability. Toward (b), recall that $\lambda^{\Dn}$   is the factor by which we rescale the distances in the tree to obtain GP convergence. So the contrary would imply that with positive probability all the distances between pair of branchpoints in the L\'evy tree would be contained into some set of the form $\frac{1}{\lambda^L}\N$, which is a.s. not the case. (c) follows by GP convergence since L\'evy trees are non-atomic. 

Hence, to show \eqref{eq:ExtractedCV}, it remains to show that we may get the next weak convergence:
\begin{equation} (\theta_i^{\D^{m_n}})_{i\geq 1} \mapsto (\theta_i)_{i\geq 1}. \label{eq:cvDeg} \end{equation}
on the topological set $(\mathcal S,\mathcal F)$ of positive decreasing sequence $(x_i)_{i\in \geq 1}$ with $\sum x_i^2\leq 1$ equipped with the product topology. It is easy to check using standard diagonal extraction procedure that  $(\mathcal S,\mathcal F)$ is compact and Polish. Thus we may choose $(m_n)_{n\in \N}$ such that $(\theta_i^{\D^{m_n}})_{i\geq 1}$ convergences weakly, and then pick  $(\theta_i)_{i\geq 1}$ as it limits. \eqref{eq:ExtractedCV} follows, which concludes the proof.
\end{proof}
\begin{remark} In \cite{ExcICRT}, Aldous, Miermont, and Pitman further conjectured that the parameters of the ICRT may be expressed in terms of the sizes of the jumps of the L\'evy processes. We believe one may compute those parameters by replacing the last step of the above proof. Indeed, instead of using a compactness argument to show \eqref{eq:cvDeg}, one may instead use the Skorokhod convergence of the Lukasiewick walk of Galton--Watson trees to deduce this convergence. (Obtaining this convergence is the usual way one starts to show convergence of Galton--Watson trees \cite{Duquesne2,Duquesne1}.)
This only gives half of the convergence of the $(\theta_i^\Dn)$, since it only describes how large degrees compare to one another, and one also need to compute the Brownian part. In other words, one also need to compute the contribution of $\sigma^\Dn =\sum_{i=1}^{\n^n} d_i^n (d_i^n-1)$ corresponding to the vertices of small degrees i.e.  to the $i\geq k_n$ where $(k_n)_{n\in \N}$ is a slowly diverging sequence. We believe that for suitable reproduction laws for the Galton--Watson trees, this can be done using standard concentration tools. The author will not pursue this direction further. 
\end{remark} 




\appendix
\section{Topological notions of convergences} \label{sec:Topology}
\subsection{Gromov--Prokhorov (GP) topology} \label{GPdef}
A measured metric space is a triple $(X,d,\mu)$ such that $(X,d)$ is a Polish space and $\mu$ is a Borel probability measure on $X$. Two such spaces $(X,d,\mu)$, $(X',d',\mu')$ are called isometry-equivalent if  there exists an isometry $f:X\to X'$ such that if $f_\star \mu$ is the image of $\mu$ by $f$ then $f_\star \mu=\mu'$. Let $\mathbb{K}_{\GP}$ be the set of isometry-equivalent classes of measured metric space. Given a measured metric space $(X,d,\mu)$, we write $[X,d,\mu]$ for the isometry-equivalence class of $(X,d,\mu)$ and frequently use the notation $X$ for either $(X,d,\mu)$ or $[X,d,\mu]$.

We now recall the definition of the Prokhorov's distance. Consider a metric space $(X,d)$. For every $A\subset X$ and $\e>0$ let $A^\e:= \{x\in X, d(x,A)<\e\}$. Then given two (Borel) probability measures $\mu$, $\nu$ on $X$, the Prokhorov distance between $\mu$ and $\nu$ is defined by 
\[d_P(\mu, \nu):= \inf\{\text{ $\e>0$: $\mu\{A\}\leq \nu \{A^\e\}$ and $\nu\{A\}\leq  \mu\{A^\e\}$, for all Borel set $A\subset X$} \}.\]

The  Gromov--Prokhorov (GP) distance is an extension of the Prokhorov's distance: For every $(X,d,\mu),(X',d',\mu')\in \K_{\GP}$ the Gromov--Prokhorov distance between $X$ and $X'$ is defined by
\[ d_{\GP}((X,d,\mu),(X',d',\mu')):=\inf_{S,\phi,\phi'} d_P(\phi_\star \mu, \phi'_\star\mu'),\]
where the infimum is taken over all metric spaces $S$ and isometric embeddings $\phi :X\to S$, $\phi' :X'\to S$. $d_{\GP}$ is indeed a distance on $\K_{\GP}$ and $(\K_{\GP},d_{\GP})$ is a Polish space (see e.g. \cite{GHP}).

We use another convenient characterization of the GP topology which relies on convergence of distance matrices: For every measured metric space  $(X,d^X,\mu^X)$ let $(x_i^X)_{i\in \N}$ be a sequence of i.i.d. random variables of common distribution $\mu^X$ and let $M^X:=(d^X(x_i^X,x_j^X))_{(i,j)\in \N^2}$. We have the following result from \cite{EquivGP},

\begin{lemma} \label{equivGP} Let $(X^n)_{n\in \N} \in \K_{\GP}^\N$ and let $X\in \K_{\GP}$ then $X^n\limit^{\GP}X$ as $n\to \infty$ if and only if $M^{X^n}$ converges in distribution toward $M^X$.
\end{lemma}

 For convenience issue we use the following extension of Lemma \ref{equivGP}. 
 \begin{lemma} \label{equivGP2} Let $(X^n)_{n\in \N}\in \K_{\GP}^\N$ and let $X\in \K_{\GP}$. Let $(y^X_i)_{i\in \N}$ be a sequence of random variables on $X$ and let $N^X:= (d^X(y_i^X,y_j^X))_{(i,j)\in \N^2}$. If
\[ M^{X_n} \limit^{(d)} N^X \quad \text{and} \quad \frac{1}{n}\sum_{i=1}^n \delta_{y^X_i} \limit^{(d)} \mu^X, \]
then $X^n \limit^{\GP} X$ and thus $M^X$ and $N^X$ have the same distribution.
\end{lemma}
\begin{proof} Fix $k \leq m\in \N$. Let $(A_1,\dots, A_k)$ be a uniform tuple of $k$ different integers in $\{1,\dots, m\}$. Since $M^{X_n} \to^{(d)} N^X$, we have 
\[ \left (d^{X^n}\left (x_i^{X^n},x_j^{X^n} \right ) \right )_{1\leq i, j \leq k} =^{(d)} \left (d^{X^n}\left (x_{A_i}^{X^n},x_{A_j}^{X^n} \right ) \right )_{1\leq i, j \leq k}
\limit^{(d)} \left (d^{X}\left (y_{A_i}^{X},y_{A_j}^{X} \right ) \right )_{1\leq i, j \leq k}.  \]
Now since as $m\to \infty$, $\frac{1}{m}\sum_{i=1}^m \delta_{y^X_i} \to \mu^X$, taking $m\to +\infty$ in the above equation yields
\[\left (d^{X^n} \left(x_i^{X^n},x_j^{X^n} \right ) \right )_{1\leq i, j \leq k} \limit^{(d)} \left (d^{X}\left (x_{i}^{X},x_{j}^{X} \right ) \right )_{1\leq i, j \leq k}.\]
Finally, since $k$ is arbitrary, Lemma \ref{equivGP} concludes the proof.
\end{proof}
\subsection{Gromov--Hausdorff (GH) topology} \label{GH}
Let $\K_{\GH}$ be the set of isometry-equivalent classes of compact metric space. For every metric space $(X,d)$, we write $[X,d]$ for the isometry-equivalent class of $(X,d)$, and frequently use the notation $X$ for either $(X,d)$ or $[X,d]$. 

 For every metric space $(X,d)$, the Hausdorff distance between $A,B\subset X$ is given by
\[d_H(A,B):= \inf\{\e>0, A\subset B^\e, B\subset A^\e \}. \]
The Gromov--Hausdorff distance between $(X,d)$,$(X',d')\in \K_{\GH}$ is given by 
\[ d_{\GH}((X,d),(X',d')):=\inf_{S,\phi,\phi'} \left (d_H(\phi(X), \phi'(X')) \right ),\]
where the infimum is taken over all metric spaces $S$ and isometric embeddings $\phi :X\to S$, $\phi' :X'\to S$. $d_{\GH}$ is indeed a distance on $\K_{\GH}$ and $(\K_{\GH},d_{\GH})$ is a Polish space. (see e.g.  \cite{GHP})

\subsection{Gromov--Hausdorff--Prokhorov (GHP) topology} \label{GHPdef}
Let $\K_{\GHP}\subset\K_{\GP}$ be the set of isometry-equivalent classes of compact measured metric space.
 The Gromov--Hausdorff--Prokhorov distance between $(X,d,\mu)$,$(X',d',\mu')\in \K_{\GHP}$ is given by 
\[ d_{\GHP}((X,d,\mu),(X',d',\mu')):=\inf_{S,\phi,\phi'} \left ( d_P(\phi_\star \mu, \phi'_\star\mu')+d_H(\phi(X), \phi'(X')) \right ),\]
where the infimum is taken over all metric spaces $S$ and isometric embeddings $\phi :X\to S$, $\phi' :X'\to S$. $d_{\GHP}$ is indeed a distance on $\K_{\GHP}$ and $(\K_{\GHP},d_{\GHP})$ is a Polish space. (see  \cite{GHP})

Note that $\GHP$ convergence implies $\GP$ convergence, then that random variables $\GHP$ measurable are also $\GH$ measurable. For every $[X,d,p]\in \K_{\GHP}$, let $[X,d]$ denote its natural projection on $\K_{\GH}$. Note that GHP convergence implies GH convergence of the projections on $\K_{\GH}$, then that the projection on $\K_{\GH}$ is a measurable function. We will need the following statement.

\begin{lemma} \label{GP+GH=GHP} Let $([X^n,d^n,p^n])_{n\in \N}$ and $[X,d,p]$ be GHP measurable random variables in $\K_{\GHP}$. Assume that almost surely $[X,d,p]$ have full support. Assume that $([X^n,d^n,p^n])_{n\in \N}$ converges weakly toward $[X,d,p]$ in a GP sens, and that $([X^n,d^n])_{n\in \N}$ converges weakly toward $[X,d]$ in a GH sens. Then $([X,d,p])_{n\in \N}$ converges weakly toward $[X,d,p]$ in a $\GHP$ sens.
\end{lemma}
\begin{proof}  Beforehand, let us introduce the covering numbers. For every metric space $(X,d)$, and $\e>0$ let $\mathcal{N}_{\e}(X,d)$ be the minimal number of closed balls of radius $\e$ to cover $X$. Note that if $(X,d)$ and $(X',d')$ are isometric spaces then for every $\e>0$, $\mathcal{N}_{\e}(X,d)=\mathcal{N}_{\e}(X',d')$, so for every $\e>0$, $\mathcal{N}_{\e}$ is well defined on $\K_{\GH}$. And, for every $\e>0$, $\mathcal{N}_{\e}$ is a measurable function on $\K_{\GH}$.
 
 It directly follows from the GH convergence that:
\begin{compactitem}
\item[(i)] The diameter of $[X^n,d^n]$ converges weakly as $n\to \infty$ toward the diameter of $[X,d]$.
\item[(ii)] For every $\e>0$, $(\mathcal{N}_{\e}[X^n,d^n])_{n\in \N}$ is tight (see Burago Burago Ivanov \cite{glue} Section 7.4).
\end{compactitem}
Hence, by \cite{GHP} Theorem 2.4, $([X^n,d^n,p^n])_{n\in \N}$ is tight for the GHP topology.

Now, let $[X',d',p']$ be a GHP subsequential limit of $([X^n,d^n,p^n])_{n\in \N}$. It is enough to show that necessarily $[X,d,p]=^{(d,\GHP)}[X',d',p']$. 

On the one hand, since $[X^n,d^n]\to^{\WGH}[X,d,p]$ and  $[X^n,d^n,p^n]\to^{\WGHP}[X,d,p]$ along a suitable sequence, we have $[X,d]=^{(d,\GH)}[X',d']$. Hence, we have
\begin{equation}  \forall \e>0, \quad \mathcal{N}_{\e}([X,d])=^{(d)}\mathcal{N}_{\e}([X',d']). \label{2209a} \end{equation}

On the other hand, since $[X^n,d^n,p^n]\to^{\WGP}[X,d,p]$ and  $[X^n,d^n,p^n]\to^{\WGHP}[X,d,p]$ along a suitable sequence, we have $[X,d,p]=^{(d,\GP)}[X',d',p']$. So, we may assume that a.s. $d_{\GP}([X,d,p],[X',d',p'])=0$. So, by definition of the GP topology, a.s. there  exists a metric space $S$ and isometric embeddings $\phi:X\to S$, $\phi':X'\to S'$ such that $d_P(\phi_\star p, \phi'_\star p')=0$, and thus $\phi_\star p=\phi'_\star p'$. Hence, if $\supp$ denotes the support of a measure, a.s. $\phi(\supp(p))=\phi'(\supp(p'))$. Therefore, since a.s. $\supp(p)=X$, we have a.s. 
\begin{equation} \forall \e>0, \quad \mathcal{N}_\e(X)=\mathcal{N}_\e(\supp(p)) =\mathcal{N}_\e(\supp(p')) \leq \mathcal{N}_\e(X'). \label{2209b} \end{equation}

Finally, note that \eqref{2209a} and \eqref{2209b} implies together that a.s. $\supp(p')=X'$. Thus, since a.s. $X=\supp(p)$ and $\phi(p)=\phi(p')$, we have a.s. $\phi(X)=\phi'(X')$. Thus, since a.s. $\phi(p)=\phi(p')$ and $\phi(X)=\phi'(X')$, we have a.s. $d_{\GHP}(X,X')=0$. 
\end{proof}

\section{Applications in some stable cases} \label{sec:stable}
In this section we briefly illustrate our main results with the $\alpha>0$ stable case: we write $\asymp$ for "up to a constant". We consider for $1\leq i\leq n$, $d_i^n\asymp \lfloor K_n/i^{1/\alpha} \rfloor$, where $K_n$ is a renormalisation factor to have $\sum d_i=n-1$. We also consider the uniform measure on the vertices. Note that $\n_0^n\to \infty$. We can distinguish several cases:

$\bullet$ \textbf{if $\alpha\geq 2$:} by standard analysis calculus we have $K_n\asymp n^{1/\alpha}$, and $\sigma^n\asymp \sqrt{n}$, for $\alpha>2$ or $\sigma^n\asymp  \sqrt{n\log(n)}$ for $\alpha=2$ so $d_1^n/n\to 0$, $\sigma_\D\gg d_1^n$, hence $\Dn \ply (1,0,0,\dots)$. In this case, Theorem \ref{D_GP_T} tells us that the $\Dn$-tree converges for the GP topology toward a ICRT of parameter $(1,0,\dots)$ which is the Brownian tree. Also since the degrees are all small, we may use the approximation $1-e^{-d^n_il/\sigma^n}\asymp d^n_i l\sigma^n$ giving $\psi^n(l)\asymp l^2$ for $l\geq 1$. Thus, Assumption \ref{D_Tight_GHP} holds, and the convergence also holds for the GHP topology. And, Theorem \ref{thm:D_Height} gives a sub-Gaussian bound for the height. The case $\alpha>2$ was already covered by Broutin and Marckert  \cite{Broutin}. 

$\bullet$ \textbf{if $1<\alpha<2$:} Again with elementary calculus, $K_n\asymp n^{1/\alpha}$, $\sigma^n\asymp n^{1/\alpha}$, and $d_1^n/n\to 0$. Now, note that $d_1/\sigma^n\nrightarrow 0$. Up to extraction, we may assume that for every $i\in \N$, $d_i/\sigma^n$ converges. So $\Dn\ply\Theta$ for some $\Theta$. And we have for $i\geq 1$, $\theta_i\asymp i^{-1/\alpha}$, so $\sum \theta_i=\infty$. Also, with elementary computations we get that as $n\to \infty$ and $k\to \infty$ slowly, $\sum_{i=1}^k d_i^n(d_i^n-1)/\sigma^n\to 1$. So by Fatou's lemma $\sum_{i=1}^\infty \theta_i^2=1$, hence $\theta_0=0$. 
Thus, Theorem \ref{D_GP_T} tells us that the $\Dn$-tree converges for the GP topology toward a $\Theta$-ICRT. And the distances are of typical order $n/\sigma^n\asymp n^{1-1/\alpha}$. Then to estimate $\psi^n(\D)$, we may use the approximation $1-e^{-d^n_il/\sigma^n}\asymp \min(1,d^n_il/\sigma^n)$, giving with integration calculus, for $l\geq 1$,
\begin{equation} \psi^n(l)\asymp l \sum_{i=1}^{n} i^{-1/\alpha}\min(1,l i^{-1/\alpha})\asymp l \sum_{i=1}^{l^{\alpha}} i^{-1/\alpha}+l^2 \sum_{i=l^{\alpha}}^{n} i^{-2/\alpha} \asymp l^{\alpha}. \label{eq:stablepsy}\end{equation}
Thus, Assumption \ref{D_Tight_GHP} holds, and the convergence holds for the GHP topology. Also, Theorem \ref{thm:D_Height} gives a bound of the form 
\[ \proba(H(T^n)>cn^{1-1/\alpha}x)\leq Ce^{-cx^{\alpha}}, \]
 which matches with Kortchemski's \cite{TailsS} bounds.

Still in this case, we may apply the results of  \cite{ICRT1} to understand the geometry of the ICRT. First,
with elementary computations we get that as $n\to \infty$ and $k\to \infty$ slowly, $\sum_{i=1}^k d_i^n(d_i^n-1)/\sigma^n\to 1$. So by Fatou's lemma $\sum_{i=1}^\infty \theta_i^2=1$, hence $\theta_0=0$.  Next we may adapt \eqref{eq:stablepsy} to get
by \eqref{eq:stablepsy} and Lemma \ref{prelimGH} (a), we also have, $ \psi^\Theta(l)=l\E[\mu^\Theta[0,l]]\asymp l^{\alpha}$.
So by \cite[Theorem 3.3]{ICRT1} the ICRT is a.s. compact. And by \cite[Theorem 3.4]{ICRT1} it has a.s. fractal dimensions $1+1/(\alpha-1)$, which again matches what is know for $\alpha$-stable trees (see \cite{Duquesne1,Duquesne2}). 

$\bullet$ \textbf{if $\alpha=1$:} This case may be treated as the case $1<\alpha<2$ but is very different: This time $K_n\asymp n/\log(n)$ and $\sigma^n\asymp n/\log(n)$. We still have convergence for the GP topology. But, when we adapt \eqref{eq:stablepsy} we get $\psi(l)\asymp l\log(l)$. And, since $\int_1^\infty 1/(l\log(l))=\infty$, Assumption \ref{D_Tight_GHP} no longer holds. Moreover, we get $\psi^\Theta(l)\asymp l\log(l)$ so by \cite[Theorem 3.3]{ICRT1}, the ICRT is a.s. not compact and the GHP convergence can never hold. Finally, Theorem \ref{thm:D_Height} upper-bound the height by 
\[ Cn/\sigma^n \int_1^{\sigma^n} \frac{dl}{\psi^\n(l)} =O(\log(n)\log\log(n)), \]
 which we believe to be sharp up to a multiplicative constant.

$\bullet$ \textbf{if $\alpha<1$:} This time $K_n\asymp n$ and $\sigma^n\sim n$ so $d_1^n/n\nrightarrow 0$. Thus, Assumption \ref{Hypo2} does not hold, and we do not have GP convergence. Instead, by Proposition \ref{P_GP_T}, we have a discrete limit toward a $\P$-tree. We may still use Theorem \ref{thm:D_Height}: We first estimate $\psi^\D(l)\asymp l$, then we get a bound of order $\log(n)$ for the height, which we think is sharp up to a multiplicative constant.

\paragraph{Acknowledgment}
I am thankful to the many people who gave feedback on a previous version of this paper. I am notably grateful to Nicolas Broutin for the supervision of my PhD, and to Svante Janson for a throughout review of a previous version.


\end{document}